\documentclass[12pt]{amsart}
\usepackage{graphicx}
\usepackage{amsmath}
\usepackage{amssymb}
\usepackage{setspace}
\newtheorem{thm}{Theorem}[section]

\newtheorem{lem}[thm]{Lemma}

\theoremstyle{definition}

\theoremstyle{remark}
\newtheorem{rem}[thm]{Remark}
\theoremstyle{conclusion}

\theoremstyle{question}
\numberwithin{equation}{section}

\begin{document}
\title[Critical order Hardy-H\'{e}non equations in $R^n$]{Liouville type theorems, a priori estimates and existence of solutions for critical order Hardy-H\'{e}non equations in $\mathbb{R}^{n}$}

\author{Wenxiong Chen, Wei Dai, Guolin Qin}

\address{Department of Mathematics, Yeshiva University, New York, NY, USA}
\email{wchen@yu.edu}

\address{School of Mathematics and Systems Science, Beihang University (BUAA), Beijing 100083, P. R. China, and LAGA, Universit\'{e} Paris 13 (UMR 7539), Paris, France}
\email{weidai@buaa.edu.cn}

\address{Institute of Applied Mathematics, Chinese Academy of Sciences, Beijing 100190, and University of Chinese Academy of Sciences, Beijing 100049, P. R. China}
\email{qinguolin18@mails.ucas.ac.cn}

\thanks{The first author is partially supported by the Simons Foundation Collaboration Grant for Mathematicians 245486. The second author is supported by the NNSF of China (No. 11501021), the Fundamental Research Funds for the Central Universities and the State Scholarship Fund of China (No. 201806025011).}

\begin{abstract}
In this paper, we consider the critical order Hardy-H\'{e}non equations
\begin{equation*}
  (-\Delta)^{\frac{n}{2}}u(x)=\frac{u^{p}(x)}{|x|^{a}}, \,\,\,\,\,\,\,\,\,\,\, x \, \in \,\, \mathbb{R}^{n},
\end{equation*}
where $n\geq4$ is even, $-\infty<a<n$, and $1<p<+\infty$. We first prove a Liouville theorem (Theorem \ref{Thm0}), that is, the unique nonnegative solution to this equation is $u\equiv0$. Then as an immediate application, we derive a priori estimates and hence existence of positive solutions to critical order Lane-Emden equations in bounded domains (Theorem \ref{Thm1} and \ref{Thm2}). Our results seem to be the first Liouville theorem, a priori estimates, and existence on the critical order equations in higher dimensions ($n\geq3$). Extensions to super-critical order Hardy-H\'{e}non equations and inequalities will also be included (Theorem \ref{Thm0-sc} and \ref{Thm1-sc}).
\end{abstract}
\maketitle {\small {\bf Keywords:} Critical order, Hardy-H\'{e}non equations, Liouville theorems, nonnegative solutions, super poly-harmonic properties, method of moving planes in a local way, blowing-up and re-scaling, a priori estimates, existence of solutions.  \\

}

\section{Introduction}

In this paper, we first investigate the uniqueness of nonnegative solutions to the following critical order Hardy-H\'{e}non equations
\begin{equation}\label{PDE}\\\begin{cases}
(-\Delta)^{\frac{n}{2}}u(x)=\frac{u^{p}(x)}{|x|^{a}} \,\,\,\,\,\,\,\,\,\, \text{in} \,\,\, \mathbb{R}^{n}, \\
u(x)\geq0, \,\,\,\,\,\,\,\, x\in\mathbb{R}^{n},
\end{cases}\end{equation}
where $u\in C^{n}(\mathbb{R}^{n})$ if $-\infty<a\leq0$, $u\in C^{n}(\mathbb{R}^{n}\setminus\{0\})\cap C^{n-2}(\mathbb{R}^{n})$ if $0<a<n$, $n\geq4$ is even, and $1<p<+\infty$.

For $0<\alpha<+\infty$, PDEs of the form
\begin{equation}\label{GPDE}
  (-\Delta)^{\frac{\alpha}{2}}u(x)=\frac{u^{p}(x)}{|x|^{a}}
\end{equation}
are called the fractional order or higher order Hardy, Lane-Emden, H\'{e}non equations for $a>0$, $a=0$, $a<0$, respectively. These equations have numerous important applications in conformal geometry and Sobolev inequalities. In particular, in the case $a=0$, (\ref{GPDE}) becomes the well-known Lane-Emden equation, which models many phenomena in mathematical physics and in astrophysics.

We say that equation \eqref{GPDE} is in critical order if $\alpha=n$, is in sub-critical order if $0<\alpha<n$ and is in super-critical order if $n<\alpha<+\infty$. Being essentially different from the sub-critical order operators, in the critical order case, the fundamental solution of $(-\Delta)^{\frac{n}{2}}$ is $c_{n}\ln\frac{1}{|x-y|}$,  which changes signs. Hence the integral representation in terms of the fundamental solution cannot be deduced directly from the super poly-harmonic properties of the solutions. Liouville type theorems for equations \eqref{GPDE} (i.e., nonexistence of nontrivial nonnegative solutions) in the whole space $\mathbb{R}^n$ and in the half space $\mathbb{R}^n_+$ have been extensively studied (see \cite{BG,CD,CFY,CL,CL1,DQ1,DQ,GS,Lin,MP,P,PS,WX} and the references therein). These Liouville theorems, in conjunction with the blowing up and re-scaling arguments, are crucial in establishing a priori estimates and hence existence of positive solutions to non-variational boundary value problems for a class of elliptic equations on bounded domains or on Riemannian manifolds with boundaries (see \cite{BM,CL3,CL4,GS1,PQS}).
\medskip

{\em (i)  Subcritical Order}
\smallskip

The results concerning equation (\ref{GPDE}) in sub-critical order are too numerous, here we only list some of them.
\smallskip

{\em (a) Subcritical nonlinearity}.

When $p<\frac{n+\alpha-2a}{n-\alpha}$, we say that the nonlinearity on the right hand side of (\ref{GPDE}) is subcritical. In this case, many nonexistence results were obtained.

For $a=0$,  $\alpha=2$,  and $1<p<\frac{n+2}{n-2}$ ($:=\infty$ if $n=2$), Liouville type theorem was established by Gidas and Spruck in their celebrated article \cite{GS}. Later, the proof was remarkably simplified by Chen and Li in \cite{CL} using the Kelvin transform and the method of moving planes (see also \cite{CL1}). For $n>\alpha=4$ and $1<p<\frac{n+4}{n-4}$, Lin \cite{Lin} proved the Liouville type theorem for all the nonnegative $C^{4}(\mathbb{R}^{n})$ smooth solutions of \eqref{GPDE}. Wei and Xu \cite{WX} generalized Lin's results to the cases when $\alpha$ is an even integer between $0$ and $n$.

For general $a\neq0$, $0<\alpha<n$, $0<p<\frac{n+\alpha-2a}{n-\alpha}$ ($1<p<+\infty$ if $\alpha=n=2$), there are also lots of literatures on Liouville type theorems for fractional order or higher order Hardy-H\'{e}non equations \eqref{GPDE}, for instance, Bidaut-V\'{e}ron and Giacomini \cite{BG}, Chen and Fang \cite{CF}, Dai and Qin \cite{DQ1}, Gidas and Spruck \cite{GS}, Mitidieri and Pohozaev \cite{MP}, Phan \cite{P}, Phan and Souplet \cite{PS}, and many others. For Liouville type theorems on systems of PDEs of type \eqref{GPDE} with respect to various types of solutions (e.g., stable, radial, nonnegative, sign-changing, $\cdots$), please refer to \cite{BG,DQ,FG,M,P,PQS,S,SZ} and the references therein.
\smallskip

{\em (b) Critical nonlinearity.}

For the critical nonlinearity cases $p=\frac{n+\alpha}{n-\alpha}$ with $a=0$ and $0<\alpha<n$, the quantitative and qualitative properties of solutions to fractional order or higher order equations \eqref{GPDE} have also been widely studied.

For integer order equations, in the special case $n>\alpha=2$, equation \eqref{GPDE} becomes the well-known Yamabe problem (for related results, please see Gidas, Ni and Nirenberg \cite{GNN1,GNN}, Caffarelli, Gidas and Spruck \cite{CGS} and the references therein). For $n>\alpha=4$, Lin \cite{Lin} classified all the positive $C^{4}$ smooth solutions of \eqref{GPDE}. In \cite{WX}, among other things, Wei and Xu proved the classification results for all the positive $C^{\alpha}$ smooth solutions of \eqref{GPDE} when $\alpha\in(0,n)$ is an even integer. For $n>\alpha=3$, Dai and Qin \cite{DQ1} classified the positive $C^{3,\epsilon}_{loc}\cap\mathcal{L}_{1}$ classical solutions of \eqref{GPDE}.

For fractional order equations, in \cite{CLO}, by developing the method of moving planes in integral forms, Chen, Li, and Ou classified all the positive $L^{\frac{2n}{n-\alpha}}_{loc}$ solutions to the equivalent integral equation of  PDE \eqref{GPDE} for general $\alpha\in(0,n)$, and as a consequence, they obtained classifications for positive weak solutions to PDE \eqref{GPDE}. Subsequently, Chen, Li, and Li \cite{CLL} developed a direct method of moving planes for fractional Laplacians $(-\Delta)^{\frac{\alpha}{2}}$ with $0<\alpha<2$,  and as an immediate application, they classified all the $C^{1,1}_{loc}\cap\mathcal{L}_{\alpha}$ positive solutions to the PDE \eqref{GPDE}, where
\begin{equation}\label{0-0}
  \mathcal{L}_{\alpha}(\mathbb{R}^{n}):=\left\{f: \mathbb{R}^{n}\rightarrow\mathbb{R}\,\Big|\,\int_{\mathbb{R}^{n}}\frac{|f(x)|}{1+|x|^{n+\alpha}}dx<\infty\right\}.
\end{equation}

{\em (ii) Critical Order}
\smallskip

In the critical order case $n=\alpha$, there have been some results on the classification of the solutions when the nonlinearity is of the form $e^{nu}$.

In \cite{CL1}, Chen and Li  classified all the $C^{2}$ smooth solutions  of the equation
\begin{equation}\label{0-1}
-\Delta u=e^{2u}, \,\,\,\,\,\,\,\, x\in\mathbb{R}^{2},
\end{equation}
with finite total volume
$$\int_{\mathbb{R}^{2}}e^{2u}dx<\infty.$$

In \cite{CY}, for general integer $n$, Chang and Yang  classified the smooth solutions to the critical order equations
\begin{equation}\label{0-3}
  (-\Delta)^{\frac{n}{2}}u=(n-1)!e^{nu}
\end{equation}
under decay conditions near infinity
$$ u(x) = \log \frac{2}{1+|x|^2} + w(\xi(x))$$
for some smooth function $w$ defined on $\mathbb{S}^n$.

When $n=4$, under weaker assumptions
$$\int_{\mathbb{R}^{4}}e^{4u}dx<\infty, \,\,\,\,\,\, u(x)=o(|x|^{2}) \,\,\,\, \text{as} \,\,\,\, |x|\rightarrow\infty,$$
Lin \cite{Lin} proved the classification results for all the $C^{4}$ smooth solutions of
$$
\Delta^{2}u=6e^{4u}, \,\,\,\,\,\,\,\, x\in\mathbb{R}^{4}. \\
$$

Then Wei and Xu \cite{WX} extended Lin's results to the cases when $n$ is an even integer.

For more literatures on the quantitative and qualitative properties of solutions to fractional order or higher order conformally invariant PDE and IE problems, please refer to \cite{CD,CL1,DFHQW,DFQ,Zhu} and the references therein.
\smallskip

{\em As for equation (\ref{GPDE}) in critical order when $\alpha =n$,
one should observe that, so far there have not seen any results.}

In this paper, we establish Liouville type theorem for nonnegative classical solutions of \eqref{PDE} in critical order cases. Our theorem seems to be the first result on this problem.

\begin{thm}\label{Thm0}
Assume $n\geq4$ is even, $-\infty<a<n$, $1<p<+\infty$ and $u$ is a nonnegative solution of \eqref{PDE}. If one of the following two assumptions
\begin{equation*}
  -\infty<a\leq2+2p \,\,\,\,\,\,\,\,\,\,\,\, \text{or} \,\,\,\,\,\,\,\,\,\,\,\, u(x)=o(|x|^{2}) \,\,\,\, \text{as} \,\, |x|\rightarrow+\infty
\end{equation*}
holds, then $u\equiv0$ in $\mathbb{R}^{n}$.
\end{thm}

\begin{rem}\label{remark0}
In Theorem \ref{Thm0}, the smoothness assumption on $u$ at $x=0$ is necessary. Equation \eqref{PDE} admits a distributional singular solution of the form $u(x)=C|x|^{-\sigma}$ with $\sigma=\frac{n-a}{p-1}>0$.
\end{rem}

\begin{rem}\label{remark1}
It is clear from the proof of Lemma \ref{lemma0} that, under the same assumptions, the super poly-harmonic properties in Lemma \ref{lemma0} also hold for nonnegative classical solutions to the following critical order Hardy-H\'{e}non type inequalities:
\begin{equation}\label{Inequality}
(-\Delta)^{\frac{n}{2}}u(x)\geq\frac{u^{p}(x)}{|x|^{a}}.
\end{equation}
Based on the super poly-harmonic properties, one can verify the proof of Liouville properties in Theorem \ref{Thm0} still work for the critical order inequalities \eqref{Inequality} (see Section 2). Thus the Liouville type results in Theorem \ref{Thm0} also hold for inequalities \eqref{Inequality} under the same assumptions.
\end{rem}

We also consider the following critical order Navier problem
\begin{equation}\label{tNavier}\\\begin{cases}
(-\Delta)^{\frac{n}{2}}u(x)=u^{p}(x)+t \,\,\,\,\,\,\,\,\,\, \text{in} \,\,\, \Omega, \\
u(x)=\Delta u(x)=\cdots=\Delta^{\frac{n}{2}-1}u(x)=0 \,\,\,\,\,\,\,\, \text{on} \,\,\, \partial\Omega,
\end{cases}\end{equation}
where $\Omega\subset\mathbb{R}^{n}$ is a bounded domain with $C^{n-2}$ boundary $\partial\Omega$, and $t$ is any nonnegative real number.

As an application of the Liouville theorem (Theorem \ref{Thm0}), we establish the following a priori estimates for all positive solutions $u$ to \eqref{tNavier} via the method of moving planes in a local way and blowing-up arguments (for related literatures on these methods, please see \cite{BC,BM,CL3,CL4,CY0,Li,SZ1}).
\begin{thm}\label{Thm1}
Assume one of the following two assumptions
\begin{equation*}
  \text{i)} \,\,\, \Omega \,\, \text{is strictly convex}, \,\, 1<p<\infty, \quad\quad\quad\, \text{or} \quad\quad\quad\, \text{ii)} \,\,\, 1<p\leq\frac{n+2}{n-2}
\end{equation*}
holds. Then, for any positive solution $u\in C^{n}(\Omega)\cap C^{n-2}(\overline{\Omega})$ to the critical order Navier problem \eqref{tNavier}, we have
\begin{equation*}
  \|u\|_{L^{\infty}(\overline{\Omega})}\leq C(n,p,t,\lambda_{1},\Omega),
\end{equation*}
where $\lambda_{1}$ is the first eigenvalue for $(-\Delta)^{\frac{n}{2}}$ in $\Omega$ with Navier boundary conditions.
\end{thm}

As a consequence of the a priori estimates (Theorem \ref{Thm1}), by applying the Leray-Schauder fixed point theorem, we derive the existence of positive solutions to the following Navier problem for critical order Lane-Emden equations
\begin{equation}\label{Navier}\\\begin{cases}
(-\Delta)^{\frac{n}{2}}u(x)=u^{p}(x) \,\,\,\,\,\,\,\,\,\, \text{in} \,\,\, \Omega, \\
u(x)=\Delta u(x)=\cdots=\Delta^{\frac{n}{2}-1}u(x)=0 \,\,\,\,\,\,\,\, \text{on} \,\,\, \partial\Omega,
\end{cases}\end{equation}
where $\Omega\subset\mathbb{R}^{n}$ is a bounded domain with $C^{n-2}$ boundary $\partial\Omega$. This seems to be the first existence result on the critical order Lane-Emden equations.
\begin{thm}\label{Thm2}
Assume one of the following two assumptions
\begin{equation*}
  \text{i)} \,\,\, \Omega \,\, \text{is strictly convex}, \,\, 1<p<\infty, \quad\quad\quad\, \text{or} \quad\quad\quad\, \text{ii)} \,\,\, 1<p\leq\frac{n+2}{n-2}
\end{equation*}
holds. Then, the critical order Navier problem \eqref{Navier} possesses at least one positive solution $u\in C^{n}(\Omega)\cap C^{n-2}(\overline{\Omega})$. Moreover, the positive solution $u$ satisfies
\begin{equation}\label{lower-bound}
  \|u\|_{L^{\infty}(\overline{\Omega})}\geq\left(\frac{\sqrt{2n}}{diam\,\Omega}\right)^{\frac{n}{p-1}}.
\end{equation}
\end{thm}

\begin{rem}\label{remark2}
The lower bounds \eqref{lower-bound} on the $L^{\infty}$ norm of positive solutions $u$ indicate that, if $diam\,\Omega<\sqrt{2n}$, then a uniform priori estimate does not exist and blow-up may occur when $p\rightarrow1+$.
\end{rem}

It's well known that the super poly-harmonic properties of solutions are crucial in establishing Liouville type theorems and the representation formulae for higher order or fractional order PDEs (see e.g. \cite{CF,CL2,WX}). In Section 2, we will first prove the super poly-harmonic properties of solutions by using ``re-centers and iteration" arguments (see Lemma \ref{lemma0}). Nevertheless, being different from the subcritical order equations, the integral representation in terms of the fundamental solution of $(-\Delta)^{\frac{n}{2}}$ can't be deduced directly from the super poly-harmonic properties, since the fundamental solution $c_{n}\ln\frac{1}{|x-y|}$ changes its signs in $\mathbb{R}^{n}$. Fortunately, based on Lemma \ref{lemma0}, we can derive instead the following integral inequality (see \eqref{formula})
\begin{equation*}
  +\infty>u(0)\geq\int_{\mathbb{R}^{n}}\frac{R_{2,n}}{|y^{\frac{n}{2}}|^{n-2}}\int_{\mathbb{R}^{n}}\frac{R_{2,n}}{|y^{\frac{n}{2}}-y^{\frac{n}{2}-1}|^{n-2}}\cdots
  \int_{\mathbb{R}^{n}}\frac{R_{2,n}}{|y^{2}-y^{1}|^{n-2}}\frac{u^{p}(y^{1})}{|y^{1}|^{a}}dy^{1}\cdots dy^{\frac{n}{2}}
\end{equation*}
for $-\infty<a<2$, where the Riesz potential's constants $R_{2,n}:=\frac{\Gamma\big(\frac{n-2}{2}\big)}{4\pi^{\frac{n}{2}}}$. This integral inequality will lead to a contradiction on integrability unless $u\equiv0$. In the case $a\geq2$, we can also obtain a contradiction using the integral estimates arguments if $u$ is not identically zero. As a consequence, Theorem \ref{Thm0} is proved.

Using similar arguments as in the proof of the above critical order problems, we will also study uniqueness of nonnegative solutions to the following super-critical order Hardy-H\'{e}non inequalities
\begin{equation}\label{Inequality-sc}\\\begin{cases}
(-\Delta)^{m}u(x)\geq\frac{u^{p}(x)}{|x|^{a}} \,\,\,\,\,\,\,\,\,\, \text{in} \,\,\, \mathbb{R}^{n}, \\
u(x)\geq0, \,\,\,\,\,\,\,\, x\in\mathbb{R}^{n},
\end{cases}\end{equation}
where $\frac{n}{2}<m<+\infty$, $n\geq2$, $-\infty<a<n$ and $1<p<+\infty$. We assume $u\in C^{2m}(\mathbb{R}^{n})$ if $-\infty<a\leq0$, $u\in C^{2m}(\mathbb{R}^{n}\setminus\{0\})\cap C^{2m-2}(\mathbb{R}^{n})$ if $0<a<n$.

It is clear from the proof of Lemma \ref{lemma0} that, under the same assumptions, the super poly-harmonic properties in Lemma \ref{lemma0} also hold for nonnegative classical solutions to super-critical order Hardy-H\'{e}non inequalities \eqref{Inequality-sc}. Based on the super poly-harmonic properties, we can derive the following Liouville type theorem on \eqref{Inequality-sc}.
\begin{thm}\label{Thm0-sc}
Under the same assumptions, the Liouville type results in Theorem \ref{Thm0} also hold for super-critical order inequalities \eqref{Inequality-sc}.
\end{thm}
\begin{rem}\label{remark5}
Based on the super poly-harmonic properties, one can verify the methods in the proof of Liouville properties in Theorem \ref{Thm0} still work for the super-critical order inequalities \eqref{Inequality-sc} (see Section 2). We only need to mention that, instead of \eqref{formula}, one can derive the following integral inequality
\begin{eqnarray*}
  +\infty&>&(-\Delta)^{m-\lceil\frac{n}{2}\rceil}u(0) \\
  &\geq&\int_{\mathbb{R}^{n}}\frac{R_{2,n}}{|y^{\lceil\frac{n}{2}\rceil}|^{n-2}}\int_{\mathbb{R}^{n}}
  \frac{R_{2,n}}{|y^{\lceil\frac{n}{2}\rceil}-y^{\lceil\frac{n}{2}\rceil-1}|^{n-2}}\cdots
  \int_{\mathbb{R}^{n}}\frac{R_{2,n}}{|y^{2}-y^{1}|^{n-2}}\frac{u^{p}(y^{1})}{|y^{1}|^{a}}dy^{1}\cdots dy^{\lceil\frac{n}{2}\rceil}
\end{eqnarray*}
for $-\infty<a<2$, where $\lceil x\rceil$ denotes the least integer $\geq x$. This integral inequality will lead to a contradiction on integrability unless $u\equiv0$. The rest of the proof are entirely similar to that of Theorem \ref{Thm0}. Thus we omit the details in the proof of Theorem \ref{Thm0-sc}.
\end{rem}

We also consider the following super-critical order Navier problem
\begin{equation}\label{tNavier-sc}\\\begin{cases}
(-\Delta)^{m}u(x)=u^{p}(x)+t \,\,\,\,\,\,\,\,\,\, \text{in} \,\,\, \Omega, \\
u(x)=\Delta u(x)=\cdots=\Delta^{m-1}u(x)=0 \,\,\,\,\,\,\,\, \text{on} \,\,\, \partial\Omega,
\end{cases}\end{equation}
where $\frac{n}{2}<m<+\infty$, $n\geq2$, $\Omega\subset\mathbb{R}^{n}$ is a bounded domain with $C^{2m-2}$ boundary $\partial\Omega$, $1<p<+\infty$ and $t$ is any nonnegative real number.

As an application of the Liouville theorem (Theorem \ref{Thm0-sc}), by using the method of moving planes in a local way and blowing-up arguments, we can establish the following a priori estimates for all positive solutions $u$ to \eqref{tNavier-sc}, and hence derive the existence of positive solutions to \eqref{tNavier-sc} with $t=0$ via the Leray-Schauder fixed point theorem.
\begin{thm}\label{Thm1-sc}
Under the same assumptions as in Theorem \ref{Thm1} and \ref{Thm2}, we have, for any positive solution $u\in C^{2m}(\Omega)\cap C^{2m-2}(\overline{\Omega})$ to \eqref{tNavier-sc},
\begin{equation*}
  \|u\|_{L^{\infty}(\overline{\Omega})}\leq C(n,m,p,t,\lambda_{1},\Omega),
\end{equation*}
where $\lambda_{1}$ is the first eigenvalue for $(-\Delta)^{m}$ in $\Omega$ with Navier boundary conditions. Furthermore, \eqref{tNavier-sc} with $t=0$ possesses at least one positive solution $u\in C^{2m}(\Omega)\cap C^{2m-2}(\overline{\Omega})$. Moreover, the positive solution $u$ satisfies
\begin{equation}\label{lower-bound-sc}
  \|u\|_{L^{\infty}(\overline{\Omega})}\geq\left(\frac{\sqrt{2n}}{diam\,\Omega}\right)^{\frac{2m}{p-1}}.
\end{equation}
\end{thm}
\begin{rem}\label{remark4}
Theorem \ref{Thm1-sc} can be proved in a similar way as Theorem \ref{Thm1} and \ref{Thm2}, so we omit the details.
\end{rem}
\begin{rem}\label{remark3}
The lower bounds \eqref{lower-bound-sc} on the $L^{\infty}$ norm of positive solutions $u$ indicate that, if $diam\,\Omega<\sqrt{2n}$, then a uniform priori estimate does not exist and blow-up may occur when $p\rightarrow1+$.
\end{rem}

This paper is organized as follows. In Section 2, we will carry out our proof of Theorem \ref{Thm0}.

In Section 3, we will derive a priori estimates for any positive solutions to the critical order Naiver problem \eqref{tNavier} (Theorem \ref{Thm1}) by applying the method of moving planes in a local way and Kelvin transforms. We will first establish a boundary layer estimates (Theorem \ref{Boundary}), in which the properties of the boundary $\partial\Omega$ play a crucial role. The global a priori estimates follows from the boundary layer estimates, blowing-up analysis, and the Liouville theorem (Theorem \ref{Thm0}).

Section 4 is devoted to the proof of Theorem \ref{Thm2}. The existence of positive solutions to the critical order Lane-Emden equations \eqref{Navier} with Navier boundary conditions will be established via the a priori estimates (Theorem \ref{Thm1}) and the Leray-Schauder fixed point theorem (Theorem \ref{L-S}). We believe that the methods in this paper can be applied to study various higher order PDEs or Systems with general nonlinear terms.

\section{Proof of Theorem \ref{Thm0}}

In this section, we will prove Theorem \ref{Thm0} by using contradiction arguments. Suppose on the contrary that $u\geq0$ satisfies equation \eqref{PDE} but $u$ is not identically zero, then there exists some $\bar{x}\in\mathbb{R}^{n}$ such that $u(\bar{x})>0$.

In the following, we will use $C$ to denote a general positive constant that may depend on $n$, $a$, $p$ and $u$, and whose value may differ from line to line.

The super poly-harmonic properties of solutions are closely related to the representation formulae and Liouville type theorems (see \cite{CF,CL2,WX} and the references therein). Therefore, in order to prove Theorem \ref{Thm0}, we first establish the following lemma.

\begin{lem}\label{lemma0}(Super poly-harmonic properties). Assume $n\geq4$ is even, $-\infty<a<n$, $1<p<+\infty$ and $u$ is a nonnegative solution of \eqref{PDE}. If one of the following two assumptions
\begin{equation*}
  -\infty<a\leq2+2p \,\,\,\,\,\,\,\,\,\,\,\, \text{or} \,\,\,\,\,\,\,\,\,\,\,\, u(x)=o(|x|^{2}) \,\,\,\, \text{as} \,\, |x|\rightarrow+\infty
\end{equation*}
holds, then
\begin{equation*}
  (-\Delta)^{i}u(x)\geq0
\end{equation*}
for every $i=1,2,\cdots,\frac{n}{2}-1$ and all $x\in\mathbb{R}^{n}$.
\end{lem}
\begin{proof}
Let $u_{i}:=(- \Delta)^{i}u$. We want to show that $u_{i}\geq0$ for $i=1,2,\cdots,\frac{n}{2}-1$. Our proof will be divided into two steps.

\textbf{\emph{Step 1.}} We first show that
\begin{equation}\label{2-1}
u_{\frac{n}{2}-1}=(-\Delta)^{\frac{n}{2}- 1}u\geq0.
\end{equation}
If not, then there exists $0\neq x^{1}\in\mathbb{R}^n$, such that
\begin{equation}\label{2-2}
  u_{\frac{n}{2}-1}(x^{1})<0.
\end{equation}

Now, let
\begin{equation}\label{2-3}
  \bar{f}(r)=\bar{f}\big(|x-x^1|\big):=\frac{1}{|\partial B_{r}(x^{1})|}\int_{\partial B_{r}(x^{1})}f(x)d\sigma
\end{equation}
be the spherical average of $f$ with respect to the center $x^1$. Then, by the well-known property $\overline{\Delta u}=\Delta\bar{u}$ and $-\infty<a<n$, we have, for any $r\geq0$ and $r\neq|x^{1}|$,
\begin{equation}\label{2-4}
\left\{{\begin{array}{l} {-\Delta\overline{u_{\frac{n}{2}-1}}(r)=\overline{\frac{u^{p}(x)}{|x|^{a}}}(r)}, \\  {} \\ {-\Delta\overline{u_{\frac{n}{2}-2}}(r)=\overline{u_{\frac{n}{2}-1}}(r)}, \\ \cdots\cdots \\ {-\Delta\overline u(r)=\overline{u_1}(r)}. \\ \end{array}}\right.
\end{equation}
From the first equation in \eqref{2-4}, by Jensen's inequality, we get, for any $r\geq0$ and $r\neq|x^{1}|$,
\begin{align}\label{2-5}
-\Delta\overline{u_{\frac{n}{2}-1}}(r)&=\frac{1}{{| {\partial
B_{r}({x^{1}})}| }}\int_{\partial B_{r}(
{x^{1}})}\frac{{u^{p}(x)}}{|x|
^{a}}d\sigma\nonumber\\
& \geq({r+| {x^{1}}| })^{-a}\frac
{1}{{| {\partial B_{r}({x^{1}})}| }}
\int_{\partial B_{r}({x^{1}})}{u^{p}(  x)}d\sigma\\
&  \geq({r+| {x^{1}}| })^{-a}\left(
{\frac{1}{{| {\partial B_{r}({x^{1}})}| }
}\int_{\partial B_{r}({x^{1}})}{u(x)
}d\sigma}\right)^{p}\nonumber\\
& =(r+|x^{1}|)^{-a}\bar{u}^{p}(r)\geq0 \quad\quad \text{if} \,\,\, 0\leq a<n,\nonumber
\end{align}
and
\begin{equation}\label{2-5'}
  -\Delta\overline{u_{\frac{n}{2}-1}}(r)\geq\big|r-|x^{1}|\big|^{-a}\bar{u}^{p}(r)\geq0 \quad\quad \text{if} \,\, -\infty<a<0.
\end{equation}
From \eqref{2-5} and \eqref{2-5'}, one has
\begin{equation}\label{2-6}
  -\frac{1}{r^{n-1}}\Big(r^{n-1}\overline{u_{\frac{n}{2}-1}}\,'(r)\Big)'\geq0.
\end{equation}
Since $-\infty<a<n$, we can integrate both sides of \eqref{2-6} from $0$ to $r$ and derive
\begin{equation}\label{2-7}
\overline{u_{\frac{n}{2}-1}}\,'(r)\leq0, \,\,\,\,\,\, \overline{u_{\frac{n}{2}-1}}(r)\leq\overline{u_{\frac{n}{2}-1}}(0)=u_{\frac{n}{2}-1}(x^{1})=:-c_{0}<0
\end{equation}
for any $r\geq0$. From the second equation in \eqref{2-4}, we deduce that
\begin{equation}\label{2-8}
-\frac{1}{{r^{n-1}}}\Big({r^{n-1}\overline{u_{\frac{n}{2}-2}}\,'}(r)\Big)'=\overline{u_{\frac{n}{2}-1}}(r)\leq-c_{0}, \,\,\,\,\,\, \forall \,\, r\geq0,
\end{equation}
integrating from $0$ to $r$ yields
\begin{equation}\label{2-9}
\overline{u_{\frac{n}{2} - 2}}\,'(r)\geq\frac{c_{0}}{n}r, \,\,\,\,\,\, \overline{u_{\frac{n}{2}-2}}(r)\geq\overline{u_{\frac{n}{2}-2}}(0)+\frac{c_{0}}{2n}r^{2}, \,\,\,\,\,\, \forall \,\, r\geq0.
\end{equation}
Hence, there exists $r_{1} > 0$ such that
\begin{equation}\label{2-10}
  \overline{u_{\frac{n}{2}-2}}(r_{1})>0.
\end{equation}
Next, take a point $x^{2}$ with $|x^{2}-x^{1}|=r_{1}$ as the new center, and make average of $\bar{f}$ at the new center $x^{2}$, i.e.,
\begin{equation}\label{2-11}
\overline{\overline{f}}(r)=\overline{\overline{f}}\big(|x-x^{2}|\big):=\frac{1}{|\partial B_{r}(x^{2})|}\int_{\partial B_{r}(x^{2})}\bar f(x)d\sigma.
\end{equation}
One can easily verify that
\begin{equation}\label{2-12}
\overline{\overline{u_{\frac{n}{2}-2}}}(0)=\overline{u _{\frac{n}{2} - 2}}(x^{2})=:c_{1}>0.
\end{equation}
Then, from \eqref{2-5} and Jensen's inequality, we deduce that $(\overline{\overline{u}},\overline{\overline{u_{1}}},\cdots,\overline{\overline{u_{\frac{n}{2}-1}}})$ satisfies
\begin{equation}\label{2-13}
\left\{{\begin{array}{l} {-\Delta\overline{\overline{u_{\frac{n}{2}-1}}}(r)=\overline{\overline{\frac{u^{p}(x)}{|x|^{a}}}}(r)\geq0}, \\
{} \\
{-\Delta\overline{\overline{u_{\frac{n}{2}-2}}}(r)=\overline{\overline{u_{\frac{n}{2}-1}}}(r)}, \\ \cdots\cdots \\ {-\Delta\overline{\overline{u}}(r)=\overline{\overline{u _1}}(r)} \\ \end{array}}\right.
\end{equation}
for any $r\geq0$. Using the same method as in obtaining the estimate \eqref{2-9}, we conclude that
\begin{equation}\label{2-14}
  \overline{\overline{u_{\frac{n}{2}-2}}}(r)\geq\overline{\overline{u_{\frac{n}{2}-2}}}(0)+\frac{{c_{0}}}{{2n}}r^{2}, \,\,\,\,\,\, \forall \,\, r\geq0.
\end{equation}
Thus we infer from \eqref{2-7}, \eqref{2-12}, \eqref{2-13} and \eqref{2-14} that
\begin{equation}\label{2-15}
\overline{\overline{u_{\frac{n}{2}-1}}}(r)\leq\overline{\overline{u_{\frac{n}{2}-1}}}(0)<0, \,\,\,\,\,\,\,\,\, \overline{\overline{u_{\frac{n}{2}-2}}}(r)\geq\overline{\overline{u_{\frac{n}{2}-2}}}(0)>0, \,\,\,\,\,\, \forall \,\, r\geq0.
\end{equation}
From the third equation in \eqref{2-13} and integrating, we derive that
\begin{equation}\label{2-16}
  \overline{\overline{u_{\frac{n}{2}-3}}}\,'(r)\leq-\frac{c_{1}}{n}r \,\,\,\,\,\, \text{and} \,\,\,\,\,\, \overline{\overline{u_{\frac{n}{2}-3}}}(r)\leq\overline{\overline{u_{\frac{n}{2}-3}}}(0)-\frac{{c_{1}}}{{2n}}r^{2}, \,\,\,\,\,\, \forall \,\, r\geq0.
\end{equation}
Hence, there exists $r_{2}>0$ such that
\begin{equation}\label{2-17}
  \overline{\overline{u_{\frac{n}{2}-3}}}(r_{2})<0.
\end{equation}
Next, we take a point $x^{3}$ with $|x^{3}-x^{2}|=r_{2}$ as the new center and make average of $\bar{\bar{f}}$ at the new center $x^{3}$, i.e.,
\begin{equation}\label{2-18}
\overline{\overline{\overline{f}}}(r)=\overline{\overline{\overline{f}}}\big(|x-x^{3}|\big):=\frac{1}{|\partial B_{r}(x^{3})|}\int_{\partial B_{r}(x^{3})}\overline{\overline{f}}(x)d\sigma.
\end{equation}
It follows that
\begin{equation}\label{2-19}
\overline{\overline{\overline{u_{\frac{n}{2}-3}}}}(0)=\overline{\overline{u _{\frac{n}{2}-3}}}(x^{3})=:-c_{2}<0.
\end{equation}
One can easily verify that $\overline{\overline{\overline{u}}}$ and $\overline{\overline{\overline{u_{i}}}}$ ($i=1,\cdots,\frac{n}{2}-1$) satisfy entirely similar equations as $(\overline{\overline{u}},\overline{\overline{u_{1}}},\cdots,\overline{\overline{u_{\frac{n}{2}-1}}})$ (see \eqref{2-13}). Using the same method in as deriving \eqref{2-15}, we arrive at
\begin{equation}\label{2-20}
\overline{\overline{\overline{u_{\frac{n}{2}-1}}}}(r)\leq\overline{\overline{\overline{u_{\frac{n}{2}-1}}}}(0)<0, \,\,\,\,\,\,\, \overline{\overline{\overline{u_{\frac{n}{2}-2}}}}(r)\geq\overline{\overline{\overline{u_{\frac{n}{2}-2}}}}(0)>0, \,\,\,\,\,\,\,
\overline{\overline{\overline{u_{\frac{n}{2}-3}}}}(r)\leq\overline{\overline{\overline{u_{\frac{n}{2}-3}}}}(0)<0
\end{equation}
for any $r\geq0$. Continuing this way, after $\frac{n}{2}$ steps of re-centers (denotes the centers by $x^{1},x^{2},\cdots,x^{\frac{n}{2}}$, the $\frac{n}{2}$ times averages of $f$ by $\widetilde{f}$ and the resulting functions coming from taking $\frac{n}{2}$ times averages by $\widetilde{u}$ and $\widetilde{u_{i}}$ for $i=1,2,\cdots,\frac{n}{2}-1$), we finally obtain that
\begin{equation}\label{2-21}
-\Delta\widetilde{u_{\frac{n}{2}-1}}(r)\geq\widetilde{\frac{u^{p}(x)}{|x|^{a}}}(r)\geq0,
\end{equation}
and for every $i=1,\cdots,\frac{n}{2}-1$,
\begin{equation}\label{2-22}
(-1)^{i}\widetilde{u_{\frac{n}{2}-i}}(r)\geq(-1)^{i}\widetilde{u_{\frac{n}{2}-i}}(0)>0, \,\,\, \,\,\, (-1)^{\frac{n}{2}}\widetilde{u}(r)\geq(-1)^{\frac{n}{2}}\widetilde{u}(0)>0, \,\,\,\,\,\, \forall \,\, r\geq0.
\end{equation}
Moreover, in the above process, we may choose $|x^{\frac{n}{2}}|$ sufficiently large, such that
\begin{equation}\label{2-100}
  |x^{\frac{n}{2}}-x^{\frac{n}{2}-1}|\geq|x^{\frac{n}{2}-1}-x^{\frac{n}{2}-2}|+\cdots+|x^{2}-x^{1}|+|x^{1}|+2.
\end{equation}

Now, if $\frac{n}{2}$ is odd, estimate \eqref{2-22} implies immediately that
\begin{equation}\label{2-23}
  \widetilde{u}(r)\leq\widetilde{u}(0)<0,
\end{equation}
which contradicts the fact that $u\geq0$. Therefore, we only need to deal with the cases that $\frac{n}{2}$ is an even integer hereafter.

Since $\frac{n}{2}$ is even, we have $\widetilde{u}(r)\geq\widetilde{u}(0)>0$ for any $r\geq0$, furthermore, one can actually observe from the above ``re-centers and iteration" process that
\begin{equation}\label{2-101}
  \widetilde{u}(0)\geq\frac{c}{2n}|x^{\frac{n}{2}}-x^{\frac{n}{2}-1}|^{2}
\end{equation}
for some constant $c>0$. Thus we may choose $|x^{\frac{n}{2}}|$ larger, such that both \eqref{2-100} and the following
\begin{equation}\label{2-102}
  \widetilde{u}(0)\geq(2p)^{\frac{np}{(p-1)^{2}}}\left(1+\frac{2n}{p}\right)^{\frac{n}{p-1}}
\end{equation}
hold.

For arbitrary $\lambda>0$, define the re-scaling of $u$ by
\begin{equation}\label{2-24}
  u_{\lambda}(x):=\lambda^{\frac{n-a}{p-1}}u(\lambda x).
\end{equation}
Then one can easily verify that equation \eqref{PDE} is invariant under this re-scaling. After $\frac{n}{2}$ steps of re-centers for $u_{\lambda}$, we denote the centers for $u_{\lambda}$ by $x_{\lambda}^{1},x_{\lambda}^{2},\cdots,x_{\lambda}^{\frac{n}{2}}$ and the resulting function coming from taking $\frac{n}{2}$ times averages by $\widetilde{u_{\lambda}}$ and $\widetilde{u_{\lambda,i}}$ for $i=1,2,\cdots,\frac{n}{2}-1$. Then \eqref{2-21} and \eqref{2-22} still hold for $(\widetilde{u_{\lambda}},\widetilde{u_{\lambda,1}},\cdots,\widetilde{u_{\lambda,\frac{n}{2}-1}})$ and $x_{\lambda}^{k}=\frac{1}{\lambda}x_{k}$ for $k=1,\cdots,\frac{n}{2}$, thus one has the following estimate
\begin{equation}\label{2-25}
  |x_{\lambda}^{\frac{n}{2}}-x_{\lambda}^{\frac{n}{2}-1}|+\cdots+|x_{\lambda}^{2}-x_{\lambda}^{1}|+|x_{\lambda}^{1}|\leq
|x^{\frac{n}{2}}-x^{\frac{n}{2}-1}|+\cdots+|x^{2}-x^{1}|+|x^{1}|=:M
\end{equation}
holds uniformly for every $\lambda\geq1$.

Since we have \eqref{2-22} and $\frac{n}{2}$ is even, it follows that
\begin{equation}\label{2-26}
  \widetilde{u}(r)\geq\widetilde{u}(0)\geq(2p)^{\frac{np}{(p-1)^{2}}}\left(1+\frac{2n}{p}\right)^{\frac{n}{p-1}}>0, \,\,\,\,\,\, \forall \,\, r\geq0,
\end{equation}
and hence
\begin{equation}\label{2-28}
\widetilde{u_{\lambda}}(r)\geq\widetilde{u_{\lambda}}(0)=\lambda^{\frac{n-a}{p-1}}\widetilde{u}(0)
\geq\lambda^{\frac{n-a}{p-1}}(2p)^{\frac{np}{(p-1)^{2}}}\left(1+\frac{2n}{p}\right)^{\frac{n}{p-1}}>0, \,\,\,\,\,\,\,\,\, \forall \,\, r\geq0.
\end{equation}
For $0\leq a<n$, by the estimate \eqref{2-28}, we may assume that, we already have
\begin{equation}\label{2-28'}
\widetilde{u}(0)\geq(1+M)^{\frac{a}{p-1}}(2p)^{\frac{np}{(p-1)^{2}}}\left(1+\frac{2n}{p}\right)^{\frac{n}{p-1}},
\end{equation}
or else we may replace $u$ by $u_{\lambda}$ with $\lambda=(1+M)^{\frac{a}{n-a}}$ (still denoted by $u$).

For any $0\leq r\leq1$, we have
\begin{equation}\label{2-27}
  \widetilde{u}(r)\geq\widetilde{u}(0)\geq l_{0}\,r^{\alpha_{0}},
\end{equation}
where
\begin{equation}\label{2-29}
  l_{0}:=\widetilde{u}(0)\geq\max\left\{(1+M)^{\frac{a}{p-1}},1\right\}(2p)^{\frac{np}{(p-1)^{2}}}\alpha_{0}^{\frac{n}{p-1}}, \,\quad\, \alpha_{0}:=\max\Big\{1,\frac{2n}{p}\Big\}\geq1.
\end{equation}
As a consequence, we infer from \eqref{2-21}, \eqref{2-100}, \eqref{2-25} and \eqref{2-27} that, for any $0\leq r\leq1$,
\begin{align}\label{2-30}
-\Delta\widetilde{u_{\frac{n}{2}-1}}(r)&\geq\Big(r+|x^{\frac{n}{2}}-x^{\frac{n}{2}-1}|+\cdots+|x^{2}-x^{1}|+|x^{1}|\Big)^{-a}\widetilde{u}^{p}(r) \nonumber \\
&\geq\big(1+M\big)^{-a}\,l_{0}^{p}\,r^{\alpha_{0}p}\\
&\geq C_{0}\,l_{0}^{p}\,r^{\alpha_{0}p} \,\,\,\,\,\,\,\,\,\,\,\,\quad \text{if} \,\, 0\leq a<n,  \nonumber
\end{align}
and
\begin{align}\label{2-30'}
-\Delta\widetilde{u_{\frac{n}{2}-1}}(r)&\geq\Big(|x^{\frac{n}{2}}-x^{\frac{n}{2}-1}|-|x^{\frac{n}{2}-1}-x^{\frac{n}{2}-2}|-\cdots-|x^{2}-x^{1}|-|x^{1}|-r\Big)^{-a}
\widetilde{u}^{p}(r) \nonumber \\
&\geq l_{0}^{p}\,r^{\alpha_{0}p}\\
&\geq C_{0}\,l_{0}^{p}\,r^{\alpha_{0}p}, \,\,\,\,\,\,\,\,\,\,\,\,\quad \text{if} \,\, -\infty<a<0,  \nonumber
\end{align}
where
\begin{equation}\label{2-99}
  C_{0}:=\min\left\{(1+M)^{-a},1\right\}\in(0,1].
\end{equation}
Integrating both sides of \eqref{2-30} and \eqref{2-30'} from $0$ to $r$ twice and taking into account of \eqref{2-22} yield
\begin{equation}\label{2-31}
  \widetilde{u_{\frac{n}{2}-1}}(r)<-\frac{C_{0}l_{0}^{p}}{(\alpha_{0}p+n)(\alpha_{0}p+2)}r^{\alpha_{0}p+2}, \,\,\,\,\,\, \forall \,\, 0\leq r\leq1.
\end{equation}
This implies
\begin{equation}\label{2-32}
  -\frac{1}{r^{n-1}}\left(r^{n-1}\widetilde{u_{\frac{n}{2}-2}}\,'(r)\right)'<-\frac{C_{0}l_{0}^{p}}{(\alpha_{0}p+n)(\alpha_{0}p+2)}r^{\alpha_{0}p+2},
\end{equation}
and consequently,
\begin{equation}\label{2-33}
  \widetilde{u_{\frac{n}{2}-2}}(r)>\frac{C_{0}l_{0}^{p}}{(\alpha_{0}p+n)(\alpha_{0}p+2)(\alpha_{0}p+n+2)(\alpha_{0}p+4)}r^{\alpha_{0}p+4}, \,\,\,\,\,\, \forall \,\, 0\leq r\leq1.
\end{equation}
Continuing this way, since $\frac{n}{2}$ is an even integer, by iteration, we can finally arrive at
\begin{equation}\label{2-34}
  \widetilde{u}(r)>\frac{C_{0}l_{0}^{p}}{(\alpha_{0}p+2n)^{n}}r^{\alpha_{0}p+n}, \,\,\,\,\,\, \forall \,\, 0\leq r\leq1.
\end{equation}
Now, define
\begin{equation}\label{2-35}
  \alpha_{k+1}:=2\alpha_{k}p\geq\alpha_{k}p+2n \,\,\,\,\,\, \text{and} \,\,\,\,\,\, l_{k+1}:=\frac{C_{0}l_{k}^{p}}{(2\alpha_{k}p)^{n}}
\end{equation}
for $k=0,1,\cdots$. Then \eqref{2-34} implies
\begin{equation}\label{2-36}
\widetilde{u}(r)>\frac{C_{0}l_{0}^{p}}{(2\alpha_{0}p)^{n}}r^{2\alpha_{0}p}=l_{1}r^{\alpha_{1}}, \,\,\,\,\, \forall \,\, r\in[0,1].
\end{equation}
Suppose we have $\widetilde{u}(r)\geq l_{k}r^{\alpha_{k}}$, then go through the entire process as above, we can derive $\widetilde{u}(r)\geq l_{k+1}r^{\alpha_{k+1}}$ for any $0\leq r\leq1$. Therefore, one can prove by induction that
\begin{equation}\label{2-37}
  \widetilde{u}(r)\geq l_{k}r^{\alpha_{k}}, \,\,\,\,\,\, \forall \,\, r\in[0,1], \,\,\,\,\,\, \forall \,\, k\in\mathbb{N}.
\end{equation}
Through direct calculations, we have
\begin{eqnarray}\label{2-38}
  l_{k}&=&\frac{C_{0}^{\frac{p^{k}-1}{p-1}}l_{0}^{p^{k}}}{(2p)^{n(k+(k-1)p+(k-2)p^{2}+\cdots+p^{k-1})}\alpha_{0}^{\frac{n(p^{k}-1)}{p-1}}} \\
 \nonumber &=& \frac{C_{0}^{\frac{p^{k}-1}{p-1}}l_{0}^{p^{k}}(2p)^{\frac{nk}{p-1}}}{(2p)^{\frac{n(p^{k+1}-p)}{(p-1)^{2}}}\alpha_{0}^{\frac{n(p^{k}-1)}{p-1}}}
\geq (2p)^{\frac{nk}{p-1}}\left(\frac{C_{0}^{\frac{1}{p-1}}l_{0}}{(2p)^{\frac{np}{(p-1)^{2}}}\alpha_{0}^{\frac{n}{p-1}}}\right)^{p^{k}}
\end{eqnarray}
for $k=0,1,2,\cdots$. From \eqref{2-29}, \eqref{2-99}, \eqref{2-37} and \eqref{2-38}, we deduce that
\begin{equation}\label{2-40}
  \widetilde{u}(1)\geq(2p)^{\frac{nk}{p-1}}\rightarrow+\infty, \,\,\,\,\,\, \text{as} \,\, k\rightarrow\infty.
\end{equation}
This is absurd. Therefore, \eqref{2-1} must hold, that is, $u_{\frac{n}{2}-1}=(-\Delta)^{\frac{n}{2}- 1}u\geq0$.

\textbf{\emph{Step 2.}} Next, we will show that all the other $u_{i}$ ($i=1,\cdots,\frac{n}{2}-2$) must be nonnegative, that is,
\begin{equation}\label{2-41}
u_{\frac{n}{2}-i}(x)\geq0, \,\,\,\,\,\,\,\,\,\,\, \forall \,\, i=2,3,\cdots,\frac{n}{2}-1, \,\,\,\,\,\, \forall \,\, x\in\mathbb{R}^{n}.
\end{equation}
Suppose on the contrary that, there exists some $2\leq i\leq\frac{n}{2}-1$ and some $x^{0}\in\mathbb{R}^{n}$ such that
\begin{equation}\label{2-42}
  u_{\frac{n}{2}-1}(x)\geq0, \,\,\,\,\, u_{\frac{n}{2}-2}(x)\geq0, \,\,\,\, \cdots, \,\,\,\, u_{\frac{n}{2}-i+1}(x)\geq0, \,\,\,\,\,\, \forall \,\, x\in\mathbb{R}^{n},
\end{equation}
\begin{equation}\label{2-43}
  u_{\frac{n}{2}-i}(x^{0})<0.
\end{equation}
Then, repeating the similar ``re-centers and iteration" arguments as in Step 1, after $\frac{n}{2}-i+1$ steps of re-centers (denotes the centers by $\bar{x}^{1},\bar{x}^{2},\cdots,\bar{x}^{\frac{n}{2}-i+1}$), the signs of the resulting functions $\widetilde{u_{\frac{n}{2}-j}}$ ($j=i,\cdots,\frac{n}{2}-1$) and $\widetilde{u}$ satisfy
\begin{equation}\label{2-44}
  (-1)^{j-i+1}\widetilde{u_{\frac{n}{2}-j}}(r)\geq(-1)^{j-i+1}\widetilde{u_{\frac{n}{2}-j}}(0)>0, \,\,\,\,\,\,
(-1)^{\frac{n}{2}-i+1}\widetilde{u}(r)\geq(-1)^{\frac{n}{2}-i+1}\widetilde{u}(0)>0
\end{equation}
for any $r\geq0$. Since $u\geq0$, it follows immediately from \eqref{2-44} that $\frac{n}{2}-i+1$ is even and
\begin{equation}\label{2-45}
  \widetilde{u}(r)\geq\widetilde{u}(0)>0, \,\,\,\,\,\, \forall \,\, r\geq0.
\end{equation}
Furthermore, since $\frac{n}{2}-i$ is odd, we infer from \eqref{2-44} that
\begin{equation}\label{2-48}
  -\Delta\widetilde{u}(r)=\widetilde{u_{1}}(r)\leq\widetilde{u_{1}}(0)=:-\widetilde{c}<0, \,\,\,\,\,\, \forall \,\, r\geq0,
\end{equation}
and hence, by integrating, one has
\begin{equation}\label{2-49}
  \widetilde{u}(r)\geq\widetilde{u}(0)+\frac{\widetilde{c}}{2n}r^{2}>\frac{\widetilde{c}}{2n}r^{2}, \,\,\,\,\,\, \forall \,\, r\geq0.
\end{equation}
Therefore, if we assume that $u(x)=o(|x|^{2})$ as $|x|\rightarrow+\infty$, we will get a contradiction from \eqref{2-49}.

Or, if we assume that $-\infty<a\leq2+2p$, combining \eqref{2-49} with the estimate \eqref{2-21}, we get that, for $r\geq r_{0}$ sufficiently large,
\begin{eqnarray}\label{2-46}
  -\Delta\widetilde{u_{\frac{n}{2}-1}}(r)&\geq&\left(r+|\bar{x}^{\frac{n}{2}-i+1}-\bar{x}^{\frac{n}{2}-i}|+\cdots+|\bar{x}^{2}-\bar{x}^{1}|+|\bar{x}^{1}|
\right)^{-a}\widetilde{u}^{p}(r) \\
  \nonumber &\geq&\left(\frac{\widetilde{c}}{4n}\right)^{p}r^{2p-a} \quad\quad\quad \text{if} \,\, 0\leq a\leq2+2p,
\end{eqnarray}
and
\begin{eqnarray}\label{2-46'}
  -\Delta\widetilde{u_{\frac{n}{2}-1}}(r)&\geq&\left(r-|\bar{x}^{\frac{n}{2}-i+1}-\bar{x}^{\frac{n}{2}-i}|-\cdots-|\bar{x}^{2}-\bar{x}^{1}|-|\bar{x}^{1}|
\right)^{-a}\widetilde{u}^{p}(r) \\
  \nonumber &\geq&\left(\frac{\widetilde{c}}{4n}\right)^{p}r^{2p-a} \quad\quad\quad \text{if} \,\, -\infty<a<0.
\end{eqnarray}
Now, by a direct integration on \eqref{2-46} and \eqref{2-46'}, we get, if $-\infty<a<2+2p$, then
\begin{equation}\label{2-47}
\widetilde{u_{\frac{n}{2}-1}}(r)\leq\widetilde{u_{\frac{n}{2}-1}}(r_{0})-\left(\frac{\widetilde{c}}{4n}\right)^{p}
\frac{r^{2+2p-a}-r_{0}^{2+2p-a}}{(n+2p-a)(2+2p-a)}\rightarrow-\infty, \,\,\,\,\,\, \text{as} \,\,\, r\rightarrow\infty;
\end{equation}
if $a=2+2p$, then
\begin{equation}\label{2-47'}
\widetilde{u_{\frac{n}{2}-1}}(r)\leq\widetilde{u_{\frac{n}{2}-1}}(r_{0})-\left(\frac{\widetilde{c}}{4n}\right)^{p}\frac{\ln r-\ln r_{0}}{n-2}\rightarrow-\infty, \,\,\,\,\,\, \text{as} \,\,\, r\rightarrow\infty.
\end{equation}
This contradicts $u_{\frac{n}{2}-1}\geq0$ and thus \eqref{2-41} must hold. This concludes the proof of Lemma \ref{lemma0}.
\end{proof}

In the following, we will continue carrying out our proof under the same assumptions as Lemma \ref{lemma0}.

By Lemma \ref{lemma0}, we can deduce from $-\Delta u\geq0$, $u\geq0$, $u(\bar{x})>0$ and maximum principle that
\begin{equation}\label{2-50}
  u(x)>0, \,\,\,\,\,\,\, \forall \,\, x\in\mathbb{R}^{n}.
\end{equation}
Then, by maximum principle, Lemma 2.1 from Chen and Lin \cite{CLin} and induction, we can also infer further from $(-\Delta)^{i} u\geq0$ ($i=1,\cdots,\frac{n}{2}-1$), $u>0$ and equation \eqref{PDE} that
\begin{equation}\label{2-51}
  (-\Delta)^{i}u(x)>0, \,\,\,\,\,\,\,\, \forall \,\, i=1,\cdots,\frac{n}{2}-1, \,\,\,\, \forall \,\, x\in\mathbb{R}^{n}.
\end{equation}

Next, we will try to obtain contradictions by discussing two different cases $-\infty<a<2$ and $a\geq2$ separately.

\emph{Case i)} $-\infty<a<2$. We will also need the following lemma concerning the removable singularity.
\begin{lem}\label{lemma1}
Suppose $u$ is harmonic in $B_{R}(0)\setminus\{0\}$ and satisfies
\begin{equation*}
  u(x)=\left\{
         \begin{array}{ll}
           o(\ln|x|), \,\,\,\,\,\,  n=2,  \\
           \\
           o(|x|^{2-n}), \,\,\,\,\,\, n\geq3,
         \end{array}
       \right. \,\,\,\,\,\,\, \text{as} \,\,\, |x|\rightarrow0.
\end{equation*}
Then $u$ can be defined at $0$ so that it is $C^{2}$ and harmonic in $B_{R}(0)$.
\end{lem}
Lemma \ref{lemma1} can be proved directly by using the Poisson integral formula and maximum principles, so we omit the details.

Now we will first show that $(-\Delta)^{\frac{n}{2}-1}u$ satisfies the following integral equation
\begin{equation}\label{2c1}
  (-\Delta)^{\frac{n}{2}-1}u(x)=\int_{\mathbb{R}^{n}}\frac{R_{2,n}}{|x-y|^{n-2}}\frac{u^{p}(y)}{|y|^{a}}dy, \,\,\,\,\,\,\,\,\,\, \forall \,\, x\in\mathbb{R}^{n},
\end{equation}
where the Riesz potential's constants $R_{\alpha,n}:=\frac{\Gamma(\frac{n-\alpha}{2})}{\pi^{\frac{n}{2}}2^{\alpha}\Gamma(\frac{\alpha}{2})}$ for $0<\alpha<n$.

To this end, for arbitrary $R>0$, let $f_{1}(u)(x):=\frac{u^{p}(x)}{|x|^{a}}$ and
\begin{equation}\label{2c2}
v_{1}^{R}(x):=\int_{B_R(0)}G_R(x,y)f_{1}(u)(y)dy,
\end{equation}
where the Green's function for $-\Delta$ on $B_R(0)$ is given by
\begin{equation}\label{Green}
  G_R(x,y)=R_{2,n}\bigg[\frac{1}{|x-y|^{n-2}}-\frac{1}{\big(|x|\cdot\big|\frac{Rx}{|x|^{2}}-\frac{y}{R}\big|\big)^{n-2}}\bigg], \,\,\,\, \text{if} \,\, x,y\in B_{R}(0),
\end{equation}
and $G_{R}(x,y)=0$ if $x$ or $y\in\mathbb{R}^{n}\setminus B_{R}(0)$.

Then, since $-\infty<a<2$, we can derive that $v_{1}^{R}\in C^{2}(\mathbb{R}^{n}\setminus\{0\})\cap C(\mathbb{R}^{n})$ and satisfies
\begin{equation}\label{2c3}\\\begin{cases}
-\Delta v_{1}^{R}(x)=\frac{u^{p}(x)}{|x|^{a}},\ \ x\in B_R(0)\setminus\{0\},\\
v_{1}^{R}(x)=0,\ \ \ \ \ \ \ x\in \mathbb{R}^{n}\setminus B_R(0).
\end{cases}\end{equation}
Let $w_{1}^R(x):=(-\Delta)^{\frac{n}{2}-1}u(x)-v_{1}^R(x)\in C^{2}(\mathbb{R}^{n}\setminus\{0\})\cap C(\mathbb{R}^{n})$. By Lemma \ref{lemma0}, Lemma \ref{lemma1}, \eqref{PDE} and \eqref{2c3}, we have $w_{1}^R\in C^{2}(\mathbb{R}^{n})$ and satisfies
\begin{equation}\label{2c4}\\\begin{cases}
-\Delta w_{1}^R(x)=0,\ \ \ \ x\in B_R(0),\\
w_{1}^{R}(x)>0, \,\,\,\,\, x\in \mathbb{R}^{n}\setminus B_R(0).
\end{cases}\end{equation}
By maximum principle, we deduce that for any $R>0$,
\begin{equation}\label{2c5}
  w_{1}^R(x)=(-\Delta)^{\frac{n}{2}-1}u(x)-v_{1}^{R}(x)>0, \,\,\,\,\,\,\, \forall \,\, x\in\mathbb{R}^{n}.
\end{equation}
Now, for each fixed $x\in\mathbb{R}^{n}$, letting $R\rightarrow\infty$ in \eqref{2c5}, we have
\begin{equation}\label{2c6}
(-\Delta)^{\frac{n}{2}-1}u(x)\geq\int_{\mathbb{R}^{n}}\frac{R_{2,n}}{|x-y|^{n-2}}f_{1}(u)(y)dy=:v_{1}(x)>0.
\end{equation}
Take $x=0$ in \eqref{2c6}, we get
\begin{equation}\label{2c7}
  \int_{\mathbb{R}^{n}}\frac{u^{p}(y)}{|y|^{n-2+a}}dy<+\infty.
\end{equation}
One can easily observe that $v_{1}\in C^{2}(\mathbb{R}^{n}\setminus\{0\})\cap C(\mathbb{R}^{n})$ is a solution of
\begin{equation}\label{2c8}
-\Delta v_{1}(x)=\frac{u^{p}(x)}{|x|^{a}},  \,\,\,\,\,\,\, x\in \mathbb{R}^n\setminus\{0\}.
\end{equation}
Define $w_{1}(x):=(-\Delta)^{\frac{n}{2}-1}u(x)-v_{1}(x)\in C^{2}(\mathbb{R}^{n}\setminus\{0\})\cap C(\mathbb{R}^{n})$. Then, by Lemma \ref{lemma1}, \eqref{PDE}, \eqref{2c6} and \eqref{2c8}, we have $w_{1}\in C^{2}(\mathbb{R}^{n})$ and satisfies
\begin{equation}\label{2c9}\\\begin{cases}
-\Delta w_{1}(x)=0, \,\,\,\,\,  x\in \mathbb{R}^n,\\
w_{1}(x)\geq0 \,\,\,\,\,\,  x\in \mathbb{R}^n.
\end{cases}\end{equation}
From Liouville theorem for harmonic functions, we can deduce that
\begin{equation}\label{2c10}
   w_{1}(x)=(-\Delta)^{\frac{n}{2}-1}u(x)-v_{1}(x)\equiv C_{1}\geq0.
\end{equation}
Therefore, we have
\begin{equation}\label{2c11}
  (-\Delta)^{\frac{n}{2}-1}u(x)=\int_{\mathbb{R}^{n}}\frac{R_{2,n}}{|x-y|^{n-2}}\frac{u^{p}(y)}{|y|^{a}}dy+C_{1}=:f_{2}(u)(x)>C_{1}\geq0.
\end{equation}

Next, for arbitrary $R>0$, let
\begin{equation}\label{2c12}
v_{2}^R(x):=\int_{B_R(0)}G_R(x,y)f_{2}(u)(y)dy.
\end{equation}
Then, we can get
\begin{equation}\label{2c13}\\\begin{cases}
-\Delta v_2^R(x)=f_{2}(u)(x),\ \ x\in B_R(0),\\
v_2^R(x)=0,\ \ \ \ \ \ \ x\in \mathbb{R}^{n}\setminus B_R(0).
\end{cases}\end{equation}
Let $w_2^R(x):=(-\Delta)^{\frac{n}{2}-2}u(x)-v_2^R(x)$. By Lemma \ref{lemma0}, \eqref{2c11} and \eqref{2c13}, we have
\begin{equation}\label{2c14}\\\begin{cases}
-\Delta w_2^R(x)=0,\ \ \ \ x\in B_R(0),\\
w_2^R(x)>0, \,\,\,\,\, x\in \mathbb{R}^{n}\setminus B_R(0).
\end{cases}\end{equation}
By maximum principle, we deduce that for any $R>0$,
\begin{equation}\label{2c15}
  w_2^R(x)=(-\Delta)^{\frac{n}{2}-2}u(x)-v_2^{R}(x)>0, \,\,\,\,\,\,\, \forall \,\, x\in\mathbb{R}^{n}.
\end{equation}
Now, for each fixed $x\in\mathbb{R}^{n}$, letting $R\rightarrow\infty$ in \eqref{2c15}, we have
\begin{equation}\label{2c16}
(-\Delta)^{\frac{n}{2}-2}u(x)\geq\int_{\mathbb{R}^{n}}\frac{R_{2,n}}{|x-y|^{n-2}}f_{2}(u)(y)dy=:v_{2}(x)>0.
\end{equation}
Take $x=0$ in \eqref{2c16}, we get
\begin{equation}\label{2c17}
  \int_{\mathbb{R}^{n}}\frac{C_{1}}{|y|^{n-2}}dy\leq\int_{\mathbb{R}^{n}}\frac{f_{2}(u)(y)}{|y|^{n-2}}dy<+\infty,
\end{equation}
it follows easily that $C_{1}=0$, and hence we have proved \eqref{2c1}, that is,
\begin{equation}\label{2c18}
  (-\Delta)^{\frac{n}{2}-1}u(x)=f_{2}(u)(x)=\int_{\mathbb{R}^{n}}\frac{R_{2,n}}{|x-y|^{n-2}}\frac{u^{p}(y)}{|y|^{a}}dy.
\end{equation}
One can easily observe that $v_{2}$ is a solution of
\begin{equation}\label{2c19}
-\Delta v_{2}(x)=f_{2}(u)(x),  \,\,\,\,\, x\in \mathbb{R}^n.
\end{equation}
Define $w_{2}(x):=(-\Delta)^{\frac{n}{2}-2}u(x)-v_{2}(x)$, then it satisfies
\begin{equation}\label{2c20}\\\begin{cases}
-\Delta w_{2}(x)=0, \,\,\,\,\,  x\in \mathbb{R}^n,\\
w_{2}(x)\geq0 \,\,\,\,\,\,  x\in\mathbb{R}^n.
\end{cases}\end{equation}
From Liouville theorem for harmonic functions, we can deduce that
\begin{equation}\label{2c21}
   w_{2}(x)=(-\Delta)^{\frac{n}{2}-2}u(x)-v_{2}(x)\equiv C_{2}\geq0.
\end{equation}
Therefore, we have proved that
\begin{equation}\label{2c22}
  (-\Delta)^{\frac{n}{2}-2}u(x)=\int_{\mathbb{R}^{n}}\frac{R_{2,n}}{|x-y|^{n-2}}f_{2}(u)(y)dy+C_{2}=:f_{3}(u)(x)>C_{2}\geq0.
\end{equation}
Through the same methods as above, we can prove that $C_{2}=0$, and hence
\begin{equation}\label{2c23}
  (-\Delta)^{\frac{n}{2}-2}u(x)=f_{3}(u)(x)=\int_{\mathbb{R}^{n}}\frac{R_{2,n}}{|x-y|^{n-2}}f_{2}(u)(y)dy.
\end{equation}
Continuing this way, defining
\begin{equation}\label{2c24}
  f_{k+1}(u)(x):=\int_{\mathbb{R}^{n}}\frac{R_{2,n}}{|x-y|^{n-2}}f_{k}(u)(y)dy
\end{equation}
for $k=1,2,\cdots,\frac{n}{2}$, then by Lemma \ref{lemma0} and induction, we have
\begin{equation}\label{2c25}
  (-\Delta)^{\frac{n}{2}-k}u(x)=f_{k+1}(u)(x)=\int_{\mathbb{R}^{n}}\frac{R_{2,n}}{|x-y|^{n-2}}f_{k}(u)(y)dy
\end{equation}
for $k=1,2,\cdots,\frac{n}{2}-1$, and
\begin{equation}\label{2c50}
  u(x)\geq f_{\frac{n}{2}+1}(u)(x)=\int_{\mathbb{R}^{n}}\frac{R_{2,n}}{|x-y|^{n-2}}f_{\frac{n}{2}}(u)(y)dy.
\end{equation}
In particular, it follows from \eqref{2c25} and \eqref{2c50} that
\begin{eqnarray}\label{2c51}
  && +\infty>(-\Delta)^{\frac{n}{2}-k}u(0)=\int_{\mathbb{R}^{n}}\frac{R_{2,n}}{|y|^{n-2}}f_{k}(u)(y)dy \\
 \nonumber &\geq& \int_{\mathbb{R}^{n}}\frac{R_{2,n}}{|y^{k}|^{n-2}}\int_{\mathbb{R}^{n}}\frac{R_{2,n}}{|y^{k}-y^{k-1}|^{n-2}}\cdots
  \int_{\mathbb{R}^{n}}\frac{R_{2,n}}{|y^{2}-y^{1}|^{n-2}}\frac{u^{p}(y^{1})}{|y^{1}|^{a}}dy^{1}dy^{2}\cdots dy^{k}
\end{eqnarray}
for $k=1,2,\cdots,\frac{n}{2}-1$, and
\begin{eqnarray}\label{formula}
  && +\infty>u(0)\geq\int_{\mathbb{R}^{n}}\frac{R_{2,n}}{|y|^{n-2}}f_{\frac{n}{2}}(u)(y)dy \\
 \nonumber &\geq& \int_{\mathbb{R}^{n}}\frac{R_{2,n}}{|y^{\frac{n}{2}}|^{n-2}}\left(\int_{\mathbb{R}^{n}}\frac{R_{2,n}}{|y^{\frac{n}{2}}-y^{\frac{n}{2}-1}|^{n-2}}\cdots
  \int_{\mathbb{R}^{n}}\frac{R_{2,n}}{|y^{2}-y^{1}|^{n-2}}\frac{u^{p}(y^{1})}{|y^{1}|^{a}}dy^{1}dy^{2}\cdots dy^{\frac{n}{2}-1}\right)dy^{\frac{n}{2}}.
\end{eqnarray}
From the properties of Riesz potential, we have the following formula (see \cite{Stein}), that is, for any $\alpha_{1},\alpha_{2}\in(0,n)$ such that $\alpha_{1}+\alpha_{2}\in(0,n)$, one has
\begin{equation}\label{2c26}
  \int_{\mathbb{R}^{n}}\frac{R_{\alpha_{1},n}}{|x-y|^{n-\alpha_{1}}}\cdot\frac{R_{\alpha_{2},n}}{|y-z|^{n-\alpha_{2}}}dy
=\frac{R_{\alpha_{1}+\alpha_{2},n}}{|x-z|^{n-(\alpha_{1}+\alpha_{2})}}.
\end{equation}
Now, by applying the formula \eqref{2c26} and direct calculations, we obtain that
\begin{eqnarray}\label{2c27}
  && \int_{\mathbb{R}^{n}}\frac{R_{2,n}}{|y^{\frac{n}{2}}-y^{\frac{n}{2}-1}|^{n-2}}\cdots
\int_{\mathbb{R}^{n}}\frac{R_{2,n}}{|y^{3}-y^{2}|^{n-2}}\cdot\frac{R_{2,n}}{|y^{2}-y^{1}|^{n-2}}dy^{2}\cdots dy^{\frac{n}{2}-1} \\
 \nonumber &=& \frac{R_{n-2,n}}{|y^{\frac{n}{2}}-y^{1}|^{2}}.
\end{eqnarray}

Now, we can deduce from \eqref{formula}, \eqref{2c27} and Fubini's theorem that
\begin{eqnarray}\label{contradiction}
  +\infty&>&u(0)\geq\int_{\mathbb{R}^{n}}\frac{R_{2,n}}{|y^{\frac{n}{2}}|^{n-2}}\left(\int_{\mathbb{R}^{n}}\frac{R_{n-2,n}}{|y^{\frac{n}{2}}-y^{1}|^{2}}\cdot
  \frac{u^{p}(y^{1})}{|y^{1}|^{a}}dy^{1}\right)dy^{\frac{n}{2}} \\
 \nonumber &=& \frac{1}{(2\pi)^{n}}\int_{\mathbb{R}^{n}}\frac{1}{|y|^{n-2}}\left(\int_{\mathbb{R}^{n}}\frac{1}{|y-z|^{2}}\cdot\frac{u^{p}(z)}{|z|^{a}}dz\right)dy.
\end{eqnarray}

We will get a contradiction from \eqref{contradiction}. Indeed, if we assume that $u$ is not identically zero, then by the integrability \eqref{2c7}, we have
\begin{equation}\label{2c60}
0<C_{0}:=\int_{\mathbb{R}^n}\frac{1}{|z|^{n-2}}\cdot\frac{u^p(z)}{|z|^a}dz<+\infty.
\end{equation}
For any given $|y|\geq 3$, if $|z|\geq\big(\ln|y|\big)^{-\frac{1}{n-2}}$, then one has immediately
\begin{equation}\label{2c61}
|y-z|\leq|y|+|z|\leq\left(|y|\big(\ln|y|\big)^{\frac{1}{n-2}}+1\right)|z|\leq 2|y|\big(\ln|y|\big)^{\frac{1}{n-2}}|z|.
\end{equation}
Thus it follows from \eqref{2c60} and \eqref{2c61} that, there exists a $R_{0}$ sufficiently large, such that, for any $|y|\geq R_{0}$, we have
\begin{eqnarray}\label{2c63}
\int_{\mathbb{R}^{n}}\frac{1}{|y-z|^2}\cdot\frac{u^p(z)}{|z|^a}dz&\geq& \frac{1}{4|y|^2\ln|y|}\int_{|z|\geq\big(\ln|y|\big)^{-\frac{1}{n-2}}}\frac{1}{|z|^{n-2}}\cdot\frac{u^p(z)}{|z|^a}dz \\
\nonumber &\geq&  \frac{C_{0}}{8|y|^2\ln|y|}.
\end{eqnarray}

Therefore, we can finally deduce from \eqref{contradiction} and \eqref{2c63} that
\begin{equation}\label{final}
 +\infty>u(0)\geq\frac{C_{0}}{8(2\pi)^{n}}\int_{|y|\geq R_{0}}\frac{1}{|y|^{n}\ln|y|}dy=+\infty,
\end{equation}
which is a contradiction! Therefore, we must have $u\equiv0$ in $\mathbb{R}^{n}$.

\emph{Case ii)} $a\geq2$. From Lemma \ref{lemma1} and \eqref{PDE}, we derive that $u\in C^{n}(\mathbb{R}^{n}\setminus\{0\})\cap C^{n-2}(\mathbb{R}^{n})$ and $u_{i}=(-\Delta)^{i} u\in C^{n-2i}(\mathbb{R}^{n}\setminus\{0\})\cap C^{n-2-2i}(\mathbb{R}^{n})$ ($i=1,\cdots,\frac{n}{2}-1$) form a positive solution to the following Lane-Emden-Hardy system
\begin{equation}\label{PDES}
\left\{{\begin{array}{l} {-\Delta u_{\frac{n}{2}-1}(x)=\frac{u^{p}(x)}{|x|^{a}}, \,\,\,\,\,\, x\in\mathbb{R}^{n}\setminus\{0\}},\\  {} \\ {-\Delta u_{\frac{n}{2}-2}(x)=u_{\frac{n}{2}-1}(x), \,\,\,\,\,\, x\in\mathbb{R}^{n}}, \\ \cdots\cdots \\ {-\Delta u(x)=u_1(x), \,\,\,\,\,\, x\in\mathbb{R}^{n}}. \\ \end{array}}\right.
\end{equation}
Since $u\in C^{n}(\mathbb{R}^{n}\setminus\{0\})\cap C^{n-2}(\mathbb{R}^{n})$, $u>0$, $u_{i}>0$, $\Delta u<0$ and $\Delta u_{i}<0$ for $|x|>0$, by direct calculations, we can deduce that
\begin{equation}\label{2-3-17}
  \frac{d}{dr}\overline{u}(r)\leq0, \,\,\,\,\,\,\,\,\, \frac{d}{dr}\overline{u_{i}}(r)\leq0 \,\,\,\,\,\,\,\,\,\,\,\, \text{for any} \,\,\, 0<r<\infty.
\end{equation}
By taking the spherical average of equations of \eqref{PDES}  with respect to the center $0$ and Jensen's inequality, we have
\begin{equation}\label{2-3-18}
  \overline{u}''(r)+\frac{n-1}{r}\overline{u}'(r)+\overline{u_{1}}(r)=0, \,\,\,\,\,\,\,\, \overline{u_{i}}''(r)+\frac{n-1}{r}\overline{u_{i}}'(r)+\overline{u_{i+1}}(r)=0,
   \,\,\,\,\,\,\,\,\, \forall \,\, r\geq0,
\end{equation}
\begin{equation}\label{2-3-19}
  \overline{u_{\frac{n}{2}-1}}''(r)+\frac{n-1}{r}\overline{u_{\frac{n}{2}-1}}'(r)+r^{-a}\overline{u}^{p}(r)\leq0, \,\,\,\,\,\,\,\,\, \forall \,\, r>0.
\end{equation}
Thus we can infer from \eqref{2-3-17} and \eqref{2-3-19} that, for any $0<r\leq1$,
\begin{equation}\label{2-3-20}
  -\overline{u_{\frac{n}{2}-1}}'(r)\geq r^{1-n}\int_{0}^{r}s^{n-1-a}\overline{u}^{p}(s)ds\geq c^pr^{1-n}\int_{0}^{r}s^{n-1-a}ds,
\end{equation}
where $c:=\min_{|x|\leq1}u(x)>0$. For $2\leq a<n$, one can deduce further from \eqref{2-3-20} that
\begin{equation}\label{2-3-21}
  -\overline{u_{\frac{n}{2}-1}}'(r)\geq \frac{c^p}{n-a}r^{1-a}, \,\,\,\,\,\,\,\, \forall \,\, 0<r\leq1.
\end{equation}
Integrating both sides of \eqref{2-3-21} from $0$ to $1$ yields that
\begin{eqnarray}\label{2-3-22}
  (-\Delta)^{\frac{n}{2}-1}u(0)=\overline{u_{\frac{n}{2}-1}}(0)&\geq&\overline{u_{\frac{n}{2}-1}}(1)
  +\frac{c^p}{n-a}\int_{0}^{1}s^{1-a}ds \\
  \nonumber &\geq& \frac{c^p}{n-a}\int_{0}^{1}s^{1-a}ds=+\infty,
\end{eqnarray}
which is a contradiction! Therefore, we must have $u\equiv0$ in $\mathbb{R}^{n}$.

This concludes the proof of Theorem \ref{Thm0}.

\section{Proof of Theorem \ref{Thm1}}
In this section, we will prove Theorem \ref{Thm1} via the method of moving planes in local way and blowing-up techniques.

\subsection{Boundary layer estimates}
In this subsection, we will first establish the following boundary layer estimates by applying Kelvin transform and the method of moving planes in local way. The properties of the boundary $\partial\Omega$ will play a crucial role in our discussions.
\begin{thm}\label{Boundary}
Assume one of the following two assumptions
\begin{equation*}
  \text{i)} \,\,\, \Omega \,\, \text{is strictly convex}, \,\, 1<p<\infty, \quad\quad\quad\, \text{or} \quad\quad\quad\, \text{ii)} \,\,\, 1<p\leq\frac{n+2}{n-2}
\end{equation*}
holds. Then, there exists a $\bar{\delta}>0$ depending only on $\Omega$ such that, for any positive solution $u\in C^{n}(\Omega)\cap C^{n-2}(\overline{\Omega})$ to the critical order Navier problem \eqref{tNavier}, we have
\begin{equation*}
  \|u\|_{L^{\infty}(\overline{\Omega}_{\bar{\delta}})}\leq C(n,p,\lambda_{1},\Omega),
\end{equation*}
where the boundary layer $\overline{\Omega}_{\bar{\delta}}:=\{x\in\overline{\Omega}\,|\,dist(x,\partial\Omega)\leq\bar{\delta}\}$.
\end{thm}
\begin{proof}
We will carry our our proof of Theorem \ref{Boundary} by discussing the two different assumptions i) and ii) separately.

\emph{Case i)} $\Omega$ is strictly convex and $1<p<\infty$. For any $x^{0}\in\partial\Omega$, let $\nu^{0}$ be the unit internal normal vector of $\partial\Omega$ at $x^{0}$, we will show that $u(x)$ is monotone increasing along the internal normal direction in the region
\begin{equation}\label{3-1}
  \overline{\Sigma_{\delta_{0}}}=\left\{x\in\overline{\Omega}\,|\,0\leq(x-x^{0})\cdot\nu^{0}\leq\delta_{0}\right\},
\end{equation}
where $\delta_{0}>0$ depends only on $x^{0}$ and $\Omega$.

To this end, we define the moving plane by
\begin{equation}\label{3-2}
  T_{\lambda}:=\{x\in\mathbb{R}^n\,|\,(x-x^{0})\cdot\nu^{0}=\lambda\},
\end{equation}
and denote
\begin{equation}\label{3-2'}
  \Sigma_{\lambda}:=\{x\in\Omega\,|\,0<(x-x^{0})\cdot\nu^{0}<\lambda\}
\end{equation}
for $\lambda>0$, and let $x^{\lambda}$ be the reflection of the point $x$ about the plane $T_{\lambda}$.

Let $u_{i}:=(-\Delta)^{i}u$ for $1\leq i\leq\frac{n}{2}-1$. By maximum principle, we have
\begin{equation}\label{3-3}
  u_{i}(x)>0 \quad\quad\, \text{in} \,\, \Omega
\end{equation}
for $1\leq i\leq\frac{n}{2}-1$. Define
\begin{equation}\label{3-4}
  U^{\lambda}(x):=u(x^{\lambda})-u(x) \quad\quad\, \text{and} \quad\quad\, U^{\lambda}_{i}(x):=u_{i}(x^{\lambda})-u_{i}(x)
\end{equation}
for $1\leq i\leq\frac{n}{2}-1$. Then we can deduce from \eqref{tNavier} that, for any $\lambda$ satisfying the reflection of $\Sigma_{\lambda}$ is contained in $\Omega$,
\begin{equation}\label{3-5}
\left\{{\begin{array}{l} {-\Delta U^{\lambda}_{\frac{n}{2}-1}(x)=u^{p}(x^{\lambda})-u^{p}(x)=p\xi^{p-1}_{\lambda}(x)U^{\lambda}(x), \,\,\,\,\,\, x\in\Sigma_{\lambda},}\\  {} \\ {-\Delta U^{\lambda}_{\frac{n}{2}-2}(x)=U^{\lambda}_{\frac{n}{2}-1}(x), \,\,\,\,\,\, x\in\Sigma_{\lambda},} \\ \cdots\cdots \\ {-\Delta U^{\lambda}(x)=U^{\lambda}_1(x), \,\,\,\,\,\, x\in\Sigma_{\lambda},} \\ {} \\
{U^{\lambda}(x)\geq0, \, U^{\lambda}_{1}(x)\geq0, \cdots, U^{\lambda}_{\frac{n}{2}-1}(x)\geq0, \,\,\,\,\,\, x\in\partial\Sigma_{\lambda},} \\ \end{array}}\right.
\end{equation}
where $\xi_{\lambda}(x)$ is valued between $u(x^{\lambda})$ and $u(x)$ by mean value theorem. Now, we will prove that there exists some $\delta>0$ sufficiently small (depending on $n$, $p$, $\|u\|_{L^{\infty}(\overline{\Omega})}$ and $\Omega$), such that
\begin{equation}\label{3-6}
  U^{\lambda}(x)\geq 0 \quad\quad\, \text{in} \,\, \Sigma_{\lambda}
\end{equation}
for all $0<\lambda\leq\delta$. This provides a starting point to move the plane $T_{\lambda}$.

Indeed, suppose on the contrary that there exists a $0<\lambda\leq\delta$ such that
\begin{equation}\label{3-7}
  U^{\lambda}(x)<0 \quad\quad\, \text{somewhere in} \,\, \Sigma_{\lambda}.
\end{equation}
Let
\begin{equation}\label{3-8}
  \beta(x):=\cos\frac{(x-x^{0})\cdot\nu^{0}}{\delta},
\end{equation}
then it follows that $\beta(x)\in[\cos1,1]$ for any $x\in\Sigma_{\lambda}$ and $-\frac{\Delta\beta(x)}{\beta(x)}=\frac{1}{\delta^2}$. Define
\begin{equation}\label{3-9}
  \overline{U^{\lambda}}(x):=\frac{U^{\lambda}(x)}{\beta(x)} \quad\quad \text{and} \quad\quad \overline{U^{\lambda}_{i}}(x):=\frac{U^{\lambda}_{i}(x)}{\beta(x)}
\end{equation}
for $i=1,\cdots,\frac{n}{2}-1$ and $x\in\Sigma_{\lambda}$. Then there exists a $x_{0}\in\Sigma_{\lambda}$ such that
\begin{equation}\label{3-10}
  \overline{U^{\lambda}}(x_{0})=\min_{\overline{\Sigma_{\lambda}}}\overline{U^{\lambda}}(x)<0.
\end{equation}
Since
\begin{equation}\label{3-11}
  -\Delta U^{\lambda}(x_{0})=-\Delta\overline{U^{\lambda}}(x_{0})\beta(x_{0})-2\nabla\overline{U^{\lambda}}(x_{0})\cdot\nabla\beta(x_{0})
  -\overline{U^{\lambda}}(x_{0})\Delta\beta(x_{0}),
\end{equation}
one immediately has
\begin{equation}\label{3-12}
  U^{\lambda}_{1}(x_{0})=-\Delta U^{\lambda}(x_{0})\leq\frac{1}{\delta^2}U^{\lambda}(x_{0})<0.
\end{equation}
Thus there exists a $x_{1}\in\Sigma_{\lambda}$ such that
\begin{equation}\label{3-13}
  \overline{U^{\lambda}_{1}}(x_{1})=\min_{\overline{\Sigma_{\lambda}}}\overline{U^{\lambda}_{1}}(x)<0.
\end{equation}
Similarly, it follows that
\begin{equation}\label{3-14}
  U^{\lambda}_{2}(x_{1})=-\Delta U^{\lambda}_{1}(x_{1})\leq\frac{1}{\delta^2}U^{\lambda}_{1}(x_{1})<0.
\end{equation}
Continuing this way, we get $\{x_{i}\}_{i=1}^{\frac{n}{2}-1}\subset\Sigma_{\lambda}$ such that
\begin{equation}\label{3-15}
  \overline{U^{\lambda}_{i}}(x_{i})=\min_{\overline{\Sigma_{\lambda}}}\overline{U^{\lambda}_{i}}(x)<0,
\end{equation}
\begin{equation}\label{3-16}
  U^{\lambda}_{i+1}(x_{i})=-\Delta U^{\lambda}_{i}(x_{i})\leq\frac{1}{\delta^2}U^{\lambda}_{i}(x_{i})<0
\end{equation}
for $i=1,2,\cdots,\frac{n}{2}-2$, and
\begin{equation}\label{3-17}
  \overline{U^{\lambda}_{\frac{n}{2}-1}}(x_{\frac{n}{2}-1})=\min_{\overline{\Sigma_{\lambda}}}\overline{U^{\lambda}_{\frac{n}{2}-1}}(x)<0,
\end{equation}
\begin{equation}\label{3-18}
  p\xi^{p-1}_{\lambda}(x_{\frac{n}{2}-1})U^{\lambda}(x_{\frac{n}{2}-1})=-\Delta U^{\lambda}_{\frac{n}{2}-1}(x_{\frac{n}{2}-1})\leq\frac{1}{\delta^2}U^{\lambda}_{\frac{n}{2}-1}(x_{\frac{n}{2}-1})<0.
\end{equation}
Therefore, we have
\begin{eqnarray}\label{3-19}
  U^{\lambda}(x_{0}) &\geq& \delta^2U^{\lambda}_{1}(x_{0})\geq \delta^{2}U^{\lambda}_{1}(x_{1})\frac{\beta(x_{0})}{\beta(x_{1})}
  \geq\delta^{4}U^{\lambda}_{2}(x_{1})\frac{\beta(x_{0})}{\beta(x_{1})} \\
 \nonumber &\geq& \delta^{4}U^{\lambda}_{2}(x_{2})\frac{\beta(x_{0})}{\beta(x_{2})}\geq\delta^{6}U^{\lambda}_{3}(x_{2})\frac{\beta(x_{0})}{\beta(x_{2})}
 \geq\delta^{6}U^{\lambda}_{3}(x_{3})\frac{\beta(x_{0})}{\beta(x_{3})} \\
 \nonumber  &\geq& \cdots\cdots\geq\delta^{n-2}U^{\lambda}_{\frac{n}{2}-1}(x_{\frac{n}{2}-1})\frac{\beta(x_{0})}{\beta(x_{\frac{n}{2}-1})} \\
 \nonumber  &\geq& p\delta^{n}\xi^{p-1}_{\lambda}(x_{\frac{n}{2}-1})U^{\lambda}(x_{\frac{n}{2}-1})\frac{\beta(x_{0})}{\beta(x_{\frac{n}{2}-1})} \\
 \nonumber  &\geq& p\delta^{n}\|u\|^{p-1}_{L^{\infty}(\overline{\Omega})}U^{\lambda}(x_{0}),
\end{eqnarray}
that is,
\begin{equation}\label{3-20}
  1\leq p\delta^{n}\|u\|^{p-1}_{L^{\infty}(\overline{\Omega})},
\end{equation}
which is absurd if we choose $\delta>0$ small enough such that
\begin{equation}\label{3-21}
  0<\delta<\left(p\|u\|^{p-1}_{L^{\infty}(\overline{\Omega})}\right)^{-\frac{1}{n}}.
\end{equation}
So far, our conclusion is: the method of moving planes can be carried on up to $\lambda=\delta$.

Next, we will move the plane $T_{\lambda}$ further along the internal normal direction at $x^{0}$ as long as the property
\begin{equation}\label{3-22}
  U^{\lambda}(x)\geq0 \quad\quad\, \text{in} \,\, \Sigma_{\lambda}
\end{equation}
holds. One can conclude that the moving planes process can be carried on (with the property \eqref{3-22}) as long as the reflection of $\overline{\Sigma_{\lambda}}$ is still contained in $\Omega$.

In fact, let $T_{\lambda_{0}}$ be a plane such that \eqref{3-22} holds and the reflection of $\overline{\Sigma_{\lambda_{0}}}$ about $T_{\lambda_{0}}$ is contained in $\Omega$. Then there exists a $\eta>0$ such that, the reflection of $\overline{\Sigma_{\lambda_{0}+\eta}}$ about $T_{\lambda_{0}+\eta}$ is still contained in $\Omega$. By \eqref{3-5}, \eqref{3-22} and strong maximum principles, one actually has
\begin{equation}\label{3-23}
  U^{\lambda_{0}}(x)>0, \quad U_{i}^{\lambda_{0}}(x)>0 \quad\quad\, \text{in} \,\, \Sigma_{\lambda_{0}},
\end{equation}
thus there exists a constant $c_{\delta}>0$ such that
\begin{equation}\label{3-24}
  U^{\lambda_{0}}(x)\geq c_{\delta}>0, \quad U_{i}^{\lambda_{0}}(x)\geq c_{\delta}>0 \quad\quad\, \text{in} \,\, \overline{\Sigma_{\lambda_{0}-\frac{\delta}{2}}}.
\end{equation}
By the smoothness of $u$, we infer that, there exists a $0<\epsilon<\min\{\eta,\frac{\delta}{2}\}$ such that, for any $\lambda\in(\lambda_{0},\lambda_{0}+\epsilon]$,
\begin{equation}\label{3-25}
  U^{\lambda}(x)>0, \quad U_{i}^{\lambda}(x)>0 \quad\quad\, \text{in} \,\, \overline{\Sigma_{\lambda_{0}-\frac{\delta}{2}}}.
\end{equation}
Suppose there exists a $\lambda_{0}<\lambda\leq\lambda_{0}+\epsilon$ such that
\begin{equation}\label{3-26}
  U^{\lambda}(x)<0 \quad\quad\, \text{somewhere in} \,\, \Sigma_{\lambda}\setminus\overline{\Sigma_{\lambda_{0}-\frac{\delta}{2}}}.
\end{equation}
Let
\begin{equation}\label{3-27}
  \overline{\beta}(x):=\cos\frac{\left(x-x^{0}-(\lambda_{0}-\frac{\delta}{2})\nu^{0}\right)\cdot\nu^{0}}{\delta} \quad\quad \text{and} \quad\quad \widetilde{U^{\lambda}}(x):=\frac{U^{\lambda}(x)}{\overline{\beta}(x)}
\end{equation}
for $x\in\Sigma_{\lambda}\setminus\overline{\Sigma_{\lambda_{0}-\frac{\delta}{2}}}$. Then there exists a $x_{0}\in\Sigma_{\lambda}\setminus\overline{\Sigma_{\lambda_{0}-\frac{\delta}{2}}}$ such that
\begin{equation}\label{3-28}
  \widetilde{U^{\lambda}}(x_{0})=\min_{\overline{\Sigma_{\lambda}\setminus\overline{\Sigma_{\lambda_{0}-\frac{\delta}{2}}}}}\widetilde{U^{\lambda}}(x)<0,
\end{equation}
by using similar arguments as proving \eqref{3-19}, one can also arrive at
\begin{equation}\label{3-29}
  U^{\lambda}(x_{0})\geq p\delta^n\|u\|^{p-1}_{L^{\infty}(\overline{\Omega})}U^{\lambda}(x_{0}),
\end{equation}
which contradicts with the choice of $\delta$. Therefore, we have proved that
\begin{equation}\label{3-30}
  U^{\lambda}(x)\geq0 \quad\quad\, \text{in} \,\, \Sigma_{\lambda}
\end{equation}
for any $\lambda\in(\lambda_{0},\lambda_{0}+\epsilon]$, that is, the plane $T_{\lambda}$ can be moved forward a little bit from $T_{\lambda_{0}}$.

Therefore, there exists a $\delta_{0}>0$ depending only on $x^{0}$ and $\Omega$ such that, $u(x)$ is monotone increasing along the internal normal direction in the region
\begin{equation}\label{3-31}
  \overline{\Sigma_{\delta_{0}}}:=\left\{x\in\overline{\Omega}\,|\,0\leq(x-x^{0})\cdot\nu^{0}\leq\delta_{0}\right\}.
\end{equation}
Since $\partial\Omega$ is $C^{n-2}$, there exists a small $0<r_{0}<\frac{\delta_{0}}{8}$ depending on $x^{0}$ and $\Omega$ such that, for any $x\in B_{r_{0}}(x^{0})\cap\partial\Omega$, $u(x)$ is monotone increasing along the internal normal direction at $x$ in the region
\begin{equation}\label{3-32}
  \overline{\Sigma_{x}}:=\left\{z\in\overline{\Omega}\,\Big|\,0\leq(z-x)\cdot\nu_{x}\leq\frac{3}{4}\delta_{0}\right\}.
\end{equation}
where $\nu_{x}$ denotes the unit internal normal vector at the point $x$ ($\nu_{x^{0}}:=\nu^{0}$). Since $\Omega$ is strictly convex, there also exists a $\theta>0$ depending on $x^{0}$ and $\Omega$ such that
\begin{equation}\label{3-33}
  I:=\left\{\nu\in\mathbb{R}^n\,|\,|\nu|=1, \, \nu\cdot\nu^{0}\geq\cos\theta\right\}\subset\left\{\nu_{x}\,|\,x\in B_{r_{0}}(x^{0})\cap\partial\Omega\right\},
\end{equation}
and hence, we have, for any $x\in B_{r_{0}}(x^{0})\cap\partial\Omega$ and $\nu\in I$,
\begin{equation}\label{3-34}
  u(x+s\nu) \quad \text{is monotone increasing with respect to} \,\,\, s\in\left[0,\frac{\delta_{0}}{2}\right].
\end{equation}
Let
\begin{equation}\label{3-35}
  D:=\{x+r_{0}\nu^{0}\,|\,x\in B_{r_{0}}(x^{0})\cap\partial\Omega\},
\end{equation}
one can easily verify that
\begin{equation}\label{3-36}
  \max_{\overline{B_{r_{0}}(x^{0})\cap\Omega}}u(x)\leq\max_{\overline{D}}u(x).
\end{equation}
For any $x\in\overline{D}$, let
\begin{equation}\label{3-37}
  \overline{V_{x}}:=\left\{x+\nu\,\Big|\,\nu\cdot\nu^{0}\geq|\nu|\cos\theta, \, |\nu|\leq\frac{\delta_{0}}{4}\right\}
\end{equation}
be a piece of cone with vertex at $x$, then it is easy to see that
\begin{equation}\label{3-38}
  u(x)=\min_{z\in\overline{V_{x}}}u(z).
\end{equation}

Now we need the following Lemma to control the integral of $u$ on $\overline{V_{x}}$.
\begin{lem}\label{lemma2}
Let $\lambda_{1}$ be the first eigenvalue for $(-\Delta)^{\frac{n}{2}}$ in $\Omega$ with Navier boundary condition, and $0<\phi\in C^n(\Omega)\cap C^{n-2}(\overline{\Omega})$ be the corresponding eigenfunction (without loss of generality, we may assume $\|\phi\|_{L^{\infty}(\overline{\Omega})}=1$), i.e.,
\begin{equation*}\\\begin{cases}
(-\Delta)^{\frac{n}{2}}\phi(x)=\lambda_{1}\phi(x) \,\,\,\,\,\,\,\,\,\, \text{in} \,\,\, \Omega, \\
\phi(x)=-\Delta \phi(x)=\cdots=(-\Delta)^{\frac{n}{2}-1}\phi(x)=0 \,\,\,\,\,\,\,\, \text{on} \,\,\, \partial\Omega.
\end{cases}\end{equation*}
Then, we have
\begin{equation*}
  \int_{\Omega}u^{p}(x)\phi(x)dx\leq C(\lambda_{1},p,|\Omega|).
\end{equation*}
\end{lem}
\begin{proof}
Multiply both side of \eqref{tNavier} by the eigenfunction $\phi(x)$ and integrate by parts, one gets
\begin{eqnarray}\label{3-39}
  \int_{\Omega}u^{p}(x)\phi(x)dx&\leq&\int_{\Omega}(u^{p}(x)+t)\phi(x)dx=\int_{\Omega}(-\Delta)^{\frac{n}{2}}u(x)\cdot\phi(x)dx \\
 \nonumber &=&\int_{\Omega}u(x)\cdot(-\Delta)^{\frac{n}{2}}\phi(x)dx=\lambda_{1}\int_{\Omega}u(x)\phi(x)dx.
\end{eqnarray}
By H\"{o}lder's inequality, we have
\begin{equation}\label{3-40}
  \int_{\Omega}u^p(x)\phi(x)dx\leq\lambda_{1}\left(\int_{\Omega}u^{p}(x)\phi(x)dx\right)^{\frac{1}{p}}\left(\int_{\Omega}\phi(x)dx\right)^{\frac{1}{p'}},
\end{equation}
and hence
\begin{equation}\label{3-41}
  \int_{\Omega}u^{p}(x)\phi(x)dx\leq\lambda^{p'}_{1}\int_{\Omega}\phi(x)dx\leq\lambda^{p'}_{1}|\Omega|.
\end{equation}
This completes the proof of Lemma \ref{lemma2}.
\end{proof}

By \eqref{3-38} and Lemma \ref{lemma2}, we see that, for any $x\in\overline{D}$,
\begin{eqnarray}\label{3-42}
  C(\lambda_{1},p,\Omega) &\geq& \int_{\Omega}u^{p}(x)\phi(x)dx\geq\int_{\overline{V_{x}}}u^{p}(z)\phi(z)dz \\
 \nonumber &\geq& u^p(x)|\overline{V_{x}}|\cdot\min_{\overline{\Omega^{r_{0}/2}}}\phi=:u^{p}(x)\cdot C(n,x^{0},\Omega),
\end{eqnarray}
where $\overline{\Omega^{r_{0}/2}}:=\{x\in\Omega\,|\,dist(x,\partial\Omega)\geq \frac{r_{0}}{2}\}$, and hence
\begin{equation}\label{3-43}
  u(x)\leq C(n,p,x^{0},\lambda_{1},\Omega), \quad\quad \forall x\in\overline{D}.
\end{equation}
Therefore, we arrive at
\begin{equation}\label{3-44}
  \max_{\overline{B_{r_{0}}(x^{0})\cap\Omega}}u(x)\leq\max_{\overline{D}}u(x)\leq C(n,p,x^{0},\lambda_{1},\Omega).
\end{equation}

Since $x^{0}\in\partial\Omega$ is arbitrary and $\partial\Omega$ is compact, we can cover $\partial\Omega$ by finite balls $\{B_{r_{k}}(x^{k})\}_{k=0}^{K}$ with centers $\{x^{k}\}_{k=0}^{K}\subset\partial\Omega$ ($K$ depends only on $\Omega$). Therefore, there exists a $\bar{\delta}>0$ depending only on $\Omega$ such that
\begin{equation}\label{3-45}
  \|u\|_{L^{\infty}(\overline{\Omega}_{\bar{\delta}})}\leq\max_{0\leq k\leq K}\max_{\overline{B_{r_{k}}(x^{k})\cap\Omega}}u(x)\leq\max_{0\leq k\leq K}C(n,p,x^{k},\lambda_{1},\Omega)=:C(n,p,\lambda_{1},\Omega),
\end{equation}
where the boundary layer $\overline{\Omega}_{\bar{\delta}}:=\{x\in\overline{\Omega}\,|\,dist(x,\partial\Omega)\leq\bar{\delta}\}$. This completes the proof of boundary layer estimates under assumption i).

\emph{Case ii)} $1<p\leq\frac{n+2}{n-2}$. Under this assumption, we do not require the convexity of $\Omega$ anymore. Since $\partial\Omega$ is $C^{n-2}$, there exists a $R_{0}>0$ depending only on $\Omega$ such that, for any $x^{0}\in\partial\Omega$, there exists a $\overline{x^{0}}$ satisfying $\overline{B_{R_{0}}(\overline{x^{0}})}\cap\overline{\Omega}=\{x^{0}\}$. For any $x^{0}\in\partial\Omega$, we define the Kelvin transform centered at $\overline{x^{0}}$ by
\begin{equation}\label{3-46}
  x\mapsto x^{\ast}:=\frac{x-\overline{x^{0}}}{|x-\overline{x^{0}}|^{2}}+\overline{x^{0}}, \quad\quad \Omega\rightarrow\Omega^{\ast}\subset B_{\frac{1}{R_{0}}}(\overline{x^{0}}),
\end{equation}
and hence there exists a small $0<\varepsilon_{0}<\frac{1}{100R_{0}}$ depending on $x^{0}$ and $\Omega$ such that $B_{\varepsilon_{0}}\big((\overline{x^{0}})^{\ast}\big)\cap\partial\Omega^{\ast}$ is strictly convex.

Now we define
\begin{equation}\label{3-47}
  \overline{u}(x^{\ast}):=\frac{1}{|x^{\ast}-\overline{x^{0}}|^{n-2}}u\left(\frac{x^{\ast}-\overline{x^{0}}}{|x^{\ast}-\overline{x^{0}}|^{2}}+\overline{x^{0}}\right),
\end{equation}
\begin{equation}\label{3-48}
  \overline{u_{i}}(x^{\ast}):=\frac{1}{|x^{\ast}-\overline{x^{0}}|^{n-2}}u_{i}\left(\frac{x^{\ast}-\overline{x^{0}}}{|x^{\ast}-\overline{x^{0}}|^{2}}+\overline{x^{0}}\right)
\end{equation}
for $i=1,\cdots,\frac{n}{2}-1$. Then, we have
\begin{equation}\label{3-48'}
  \overline{u}(x^{\ast})>0, \quad\, \overline{u_{i}}(x^{\ast})>0 \quad\quad \text{in} \,\, \Omega^{\ast},
\end{equation}
and from \eqref{tNavier}, we infer that $\overline{u}(x^{\ast})$ and $\overline{u_{i}}(x^{\ast})$ satisfy
\begin{equation}\label{3-49}
\left\{{\begin{array}{l} {-\Delta \overline{u_{\frac{n}{2}-1}}(x^{\ast})=\frac{1}{|x^{\ast}-\overline{x^{0}}|^{\tau}}\overline{u}^{p}(x^{\ast})+\frac{t}{|x^{\ast}-\overline{x^{0}}|^{n+2}}, \,\,\,\,\,\, x^{\ast}\in\Omega^{\ast},}\\  {} \\ {-\Delta \overline{u_{\frac{n}{2}-2}}(x^{\ast})=\frac{1}{|x^{\ast}-\overline{x^{0}}|^{4}}\overline{u_{\frac{n}{2}-1}}(x^{\ast}), \,\,\,\,\,\, x^{\ast}\in\Omega^{\ast},} \\ \cdots\cdots \\ {-\Delta \overline{u}(x^{\ast})=\frac{1}{|x^{\ast}-\overline{x^{0}}|^{4}}\overline{u_{1}}(x^{\ast}), \,\,\,\,\,\, x^{\ast}\in\Omega^{\ast},} \\ {} \\
{\overline{u}(x^{\ast})=\overline{u_{1}}(x^{\ast})=\cdots=\overline{u_{\frac{n}{2}-1}}(x^{\ast})=0, \,\,\,\,\,\, x^{\ast}\in\partial\Omega^{\ast},} \\ \end{array}}\right.
\end{equation}
where $\tau:=n+2-p(n-2)\geq0$. Let $\nu^{0}$ be the unit internal normal vector of $\partial\Omega^{\ast}$ at $(x^{0})^{\ast}$, we will show that $\overline{u}(x^{\ast})$ is monotone increasing along the internal normal direction in the region
\begin{equation}\label{3-50}
  \overline{\Sigma_{\delta_{\ast}}}=\left\{x^{\ast}\in\overline{\Omega^{\ast}}\,|\,0\leq(x^{\ast}-(x^{0})^{\ast})\cdot\nu^{0}\leq\delta_{\ast}\right\},
\end{equation}
where $\delta_{\ast}>0$ depends only on $x^{0}$ and $\Omega$.

For this purpose, we define the moving plane by
\begin{equation}\label{3-51}
  T^{\ast}_{\lambda}:=\{x^{\ast}\in\mathbb{R}^n\,|\,(x^{\ast}-(x^{0})^{\ast})\cdot\nu^{0}=\lambda\},
\end{equation}
and denote
\begin{equation}\label{3-52}
  \Sigma^{\ast}_{\lambda}:=\{x^{\ast}\in\Omega^{\ast}\,|\,0<(x^{\ast}-(x^{0})^{\ast})\cdot\nu^{0}<\lambda\}
\end{equation}
for $\lambda>0$, and let $x^{\ast}_{\lambda}$ be the reflection of the point $x^{\ast}$ about the plane $T^{\ast}_{\lambda}$.

Define
\begin{equation}\label{3-53}
  W^{\lambda}(x^{\ast}):=\overline{u}(x^{\ast}_{\lambda})-\overline{u}(x^{\ast}) \quad\quad\, \text{and} \quad\quad\, W^{\lambda}_{i}(x^{\ast}):=\overline{u_{i}}(x^{\ast}_{\lambda})-\overline{u_{i}}(x^{\ast})
\end{equation}
for $1\leq i\leq\frac{n}{2}-1$. Then we can deduce from \eqref{3-49} that, for any $\lambda$ satisfying the reflection of $\Sigma^{\ast}_{\lambda}$ is contained in $\Omega^{\ast}$,
\begin{equation}\label{3-54}
\left\{{\begin{array}{l} {-\Delta W^{\lambda}_{\frac{n}{2}-1}(x^{\ast})=\frac{\overline{u}^{p}(x^{\ast}_{\lambda})}{|x^{\ast}_{\lambda}-\overline{x^{0}}|^{\tau}}
-\frac{\overline{u}^{p}(x^{\ast})}{|x^{\ast}-\overline{x^{0}}|^{\tau}}
+\frac{t}{|x^{\ast}_{\lambda}-\overline{x^{0}}|^{n+2}}-\frac{t}{|x^{\ast}-\overline{x^{0}}|^{n+2}}, \,\,\,\,\,\, x^{\ast}\in\Sigma^{\ast}_{\lambda},}\\  {} \\ {-\Delta W^{\lambda}_{\frac{n}{2}-2}(x^{\ast})=\frac{\overline{u_{\frac{n}{2}-1}}(x^{\ast}_{\lambda})}{|x^{\ast}_{\lambda}-\overline{x^{0}}|^{4}}
-\frac{\overline{u_{\frac{n}{2}-1}}(x^{\ast})}{|x^{\ast}-\overline{x^{0}}|^{4}}, \,\,\,\,\,\, x^{\ast}\in\Sigma^{\ast}_{\lambda},} \\ \cdots\cdots \\ {-\Delta W^{\lambda}(x^{\ast})=\frac{\overline{u_{1}}(x^{\ast}_{\lambda})}{|x^{\ast}_{\lambda}-\overline{x^{0}}|^{4}}
-\frac{\overline{u_{1}}(x^{\ast})}{|x^{\ast}-\overline{x^{0}}|^{4}}, \,\,\,\,\,\, x^{\ast}\in\Sigma^{\ast}_{\lambda},} \\ {} \\
{W^{\lambda}(x^{\ast})\geq0, \, W^{\lambda}_{1}(x^{\ast})\geq0, \cdots, W^{\lambda}_{\frac{n}{2}-1}(x^{\ast})\geq0, \,\,\,\,\,\, x^{\ast}\in\partial\Sigma^{\ast}_{\lambda}.} \\ \end{array}}\right.
\end{equation}
Notice that for any $x^{\ast}\in\Sigma^{\ast}_{\lambda}$ with $\lambda<\frac{1}{R_{0}}$, one has
\begin{equation}\label{3-55}
  0<|x^{\ast}_{\lambda}-\overline{x^{0}}|<|x^{\ast}-\overline{x^{0}}|<\frac{1}{R_{0}},
\end{equation}
and hence, by direct calculations, it follows from \eqref{3-54} and $t\geq0$ that
\begin{equation}\label{3-56}
\left\{{\begin{array}{l} {-\Delta W^{\lambda}_{\frac{n}{2}-1}(x^{\ast})\geq\frac{p\varphi^{p-1}_{\lambda}(x^{\ast})}{|x^{\ast}-\overline{x^{0}}|^{\tau}}W^{\lambda}(x^{\ast})\geq pR^{\tau}_{0}\varphi^{p-1}_{\lambda}(x^{\ast})W^{\lambda}(x^{\ast}), \,\,\,\,\,\, x^{\ast}\in\Sigma^{\ast}_{\lambda},}\\  {} \\ {-\Delta W^{\lambda}_{\frac{n}{2}-2}(x^{\ast})\geq R^{4}_{0}\,W^{\lambda}_{\frac{n}{2}-1}(x^{\ast}), \,\,\,\,\,\, x^{\ast}\in\Sigma^{\ast}_{\lambda},} \\ \cdots\cdots \\ {-\Delta W^{\lambda}(x^{\ast})\geq R^{4}_{0}\,W^{\lambda}_{1}(x^{\ast}), \,\,\,\,\,\, x^{\ast}\in\Sigma^{\ast}_{\lambda},} \\ {} \\
{W^{\lambda}(x^{\ast})\geq0, \, W^{\lambda}_{1}(x^{\ast})\geq0, \cdots, W^{\lambda}_{\frac{n}{2}-1}(x^{\ast})\geq0, \,\,\,\,\,\, x^{\ast}\in\partial\Sigma^{\ast}_{\lambda}.} \\ \end{array}}\right.
\end{equation}
where $\varphi_{\lambda}(x^{\ast})$ is valued between $\overline{u}(x^{\ast}_{\lambda})$ and $\overline{u}(x^{\ast})$ by mean value theorem, and thus
\begin{equation}\label{3-57}
  \|\varphi_{\lambda}\|_{L^{\infty}(\overline{\Sigma^{\ast}_{\lambda}})}\leq\left(diam\,\Omega+R_{0}\right)^{n-2}\|u\|_{L^{\infty}(\overline{\Omega})}.
\end{equation}
Now, we will prove that there exists some $\delta>0$ sufficiently small (depending on $n$, $p$, $\|u\|_{L^{\infty}(\overline{\Omega})}$ and $\Omega$), such that
\begin{equation}\label{3-58}
  W^{\lambda}(x^{\ast})\geq 0 \quad\quad\, \text{in} \,\, \Sigma^{\ast}_{\lambda}
\end{equation}
for all $0<\lambda\leq\delta$. This provides a starting point to move the plane $T^{\ast}_{\lambda}$.

In fact, suppose on the contrary that there exists a $0<\lambda\leq\delta$ such that
\begin{equation}\label{3-59}
  W^{\lambda}(x^{\ast})<0 \quad\quad\, \text{somewhere in} \,\, \Sigma^{\ast}_{\lambda}.
\end{equation}
Let
\begin{equation}\label{3-60}
  \psi(x^{\ast}):=\cos\frac{(x^{\ast}-(x^{0})^{\ast})\cdot\nu^{0}}{\delta},
\end{equation}
then $\psi(x^{\ast})\in[\cos1,1]$ for any $x^{\ast}\in\Sigma^{\ast}_{\lambda}$ and $-\frac{\Delta\psi}{\psi}=\frac{1}{\delta^2}$. Define
\begin{equation}\label{3-61}
  \overline{W^{\lambda}}(x^{\ast}):=\frac{W^{\lambda}(x^{\ast})}{\psi(x^{\ast})} \quad\quad \text{and} \quad\quad \overline{W^{\lambda}_{i}}(x^{\ast}):=\frac{W^{\lambda}_{i}(x^{\ast})}{\psi(x^{\ast})}
\end{equation}
for $i=1,\cdots,\frac{n}{2}-1$ and $x^{\ast}\in\Sigma^{\ast}_{\lambda}$. Then there exists a $x^{\ast}_{0}\in\Sigma^{\ast}_{\lambda}$ such that
\begin{equation}\label{3-62}
  \overline{W^{\lambda}}(x^{\ast}_{0})=\min_{\overline{\Sigma^{\ast}_{\lambda}}}\overline{W^{\lambda}}(x^{\ast})<0.
\end{equation}
Since
\begin{equation}\label{3-63}
  -\Delta W^{\lambda}(x^{\ast}_{0})=-\Delta\overline{W^{\lambda}}(x^{\ast}_{0})\psi(x^{\ast}_{0})-2\nabla\overline{W^{\lambda}}(x^{\ast}_{0})\cdot\nabla\psi(x^{\ast}_{0})
  -\overline{W^{\lambda}}(x^{\ast}_{0})\Delta\psi(x^{\ast}_{0}),
\end{equation}
one immediately has
\begin{equation}\label{3-64}
  R^{4}_{0}\,W^{\lambda}_{1}(x^{\ast}_{0})\leq-\Delta W^{\lambda}(x^{\ast}_{0})\leq\frac{1}{\delta^2}W^{\lambda}(x^{\ast}_{0})<0.
\end{equation}
Thus there exists a $x^{\ast}_{1}\in\Sigma^{\ast}_{\lambda}$ such that
\begin{equation}\label{3-65}
  \overline{W^{\lambda}_{1}}(x^{\ast}_{1})=\min_{\overline{\Sigma^{\ast}_{\lambda}}}\overline{W^{\lambda}_{1}}(x^{\ast})<0.
\end{equation}
Similarly, it follows that
\begin{equation}\label{3-66}
  R^{4}_{0}\,W^{\lambda}_{2}(x^{\ast}_{1})\leq-\Delta W^{\lambda}_{1}(x^{\ast}_{1})\leq\frac{1}{\delta^2}W^{\lambda}_{1}(x^{\ast}_{1})<0.
\end{equation}
Continuing this way, we get $\{x^{\ast}_{i}\}_{i=1}^{\frac{n}{2}-1}\subset\Sigma^{\ast}_{\lambda}$ such that
\begin{equation}\label{3-67}
  \overline{W^{\lambda}_{i}}(x^{\ast}_{i})=\min_{\overline{\Sigma^{\ast}_{\lambda}}}\overline{W^{\lambda}_{i}}(x^{\ast})<0,
\end{equation}
\begin{equation}\label{3-68}
  R^{4}_{0}\,W^{\lambda}_{i+1}(x^{\ast}_{i})\leq-\Delta W^{\lambda}_{i}(x^{\ast}_{i})\leq\frac{1}{\delta^2}W^{\lambda}_{i}(x^{\ast}_{i})<0
\end{equation}
for $i=1,2,\cdots,\frac{n}{2}-2$, and
\begin{equation}\label{3-69}
  \overline{W^{\lambda}_{\frac{n}{2}-1}}(x^{\ast}_{\frac{n}{2}-1})=\min_{\overline{\Sigma^{\ast}_{\lambda}}}\overline{W^{\lambda}_{\frac{n}{2}-1}}(x^{\ast})<0,
\end{equation}
\begin{equation}\label{3-70}
  pR^{\tau}_{0}\varphi^{p-1}_{\lambda}(x^{\ast}_{\frac{n}{2}-1})W^{\lambda}(x^{\ast}_{\frac{n}{2}-1})\leq-\Delta W^{\lambda}_{\frac{n}{2}-1}(x^{\ast}_{\frac{n}{2}-1})\leq\frac{1}{\delta^2}W^{\lambda}_{\frac{n}{2}-1}(x^{\ast}_{\frac{n}{2}-1})<0.
\end{equation}
Therefore, we have
\begin{eqnarray}\label{3-71}
  W^{\lambda}(x^{\ast}_{0}) &\geq& (\delta R^{2}_{0})^{2}W^{\lambda}_{1}(x^{\ast}_{0})\geq (\delta R^{2}_{0})^{2}W^{\lambda}_{1}(x^{\ast}_{1})\frac{\psi(x^{\ast}_{0})}{\psi(x^{\ast}_{1})} \\
 \nonumber &\geq& (\delta R^{2}_{0})^{4}W^{\lambda}_{2}(x^{\ast}_{1})\frac{\psi(x^{\ast}_{0})}{\psi(x^{\ast}_{1})}
 \geq(\delta R^{2}_{0})^{4}W^{\lambda}_{2}(x^{\ast}_{2})\frac{\psi(x^{\ast}_{0})}{\psi(x^{\ast}_{2})} \\
 \nonumber &\geq& (\delta R^{2}_{0})^{6}W^{\lambda}_{3}(x^{\ast}_{2})\frac{\psi(x^{\ast}_{0})}{\psi(x^{\ast}_{2})}
 \geq(\delta R^{2}_{0})^{6}W^{\lambda}_{3}(x^{\ast}_{3})\frac{\psi(x^{\ast}_{0})}{\psi(x^{\ast}_{3})} \\
 \nonumber  &\geq& \cdots\cdots\geq(\delta R^{2}_{0})^{n-2}W^{\lambda}_{\frac{n}{2}-1}(x^{\ast}_{\frac{n}{2}-1})\frac{\psi(x^{\ast}_{0})}{\psi(x^{\ast}_{\frac{n}{2}-1})} \\
 \nonumber  &\geq& p\delta^{n}R^{n+2-(p-2)(n-2)}_{0}\varphi^{p-1}_{\lambda}(x^{\ast}_{\frac{n}{2}-1})W^{\lambda}(x^{\ast}_{\frac{n}{2}-1})
 \frac{\psi(x^{\ast}_{0})}{\psi(x^{\ast}_{\frac{n}{2}-1})} \\
 \nonumber  &\geq& p\delta^{n}R^{n+2-(p-2)(n-2)}_{0}\left(diam\,\Omega+R_{0}\right)^{(p-1)(n-2)}\|u\|^{p-1}_{L^{\infty}(\overline{\Omega})}W^{\lambda}(x^{\ast}_{0}),
\end{eqnarray}
that means,
\begin{equation}\label{3-72}
  1\leq p\delta^{n}\left(diam\,\Omega+R_{0}\right)^{2n}\|u\|^{p-1}_{L^{\infty}(\overline{\Omega})},
\end{equation}
which is absurd if we choose $\delta>0$ small enough such that
\begin{equation}\label{3-73}
  0<\delta<\left(diam\,\Omega+R_{0}\right)^{-2}\left(p\|u\|^{p-1}_{L^{\infty}(\overline{\Omega})}\right)^{-\frac{1}{n}}.
\end{equation}
So far, we have proved that the plane $T^{\ast}_{\lambda}$ can be moved on up to $\lambda=\delta$.

Next, we will move the plane $T^{\ast}_{\lambda}$ further along the internal normal direction at $(x^{0})^{\ast}$ as long as the property
\begin{equation}\label{3-74}
  W^{\lambda}(x^{\ast})\geq0 \quad\quad\, \text{in} \,\, \Sigma^{\ast}_{\lambda}
\end{equation}
holds. Completely similar to the proof of \emph{Case i)}, one can actually show that the method of moving planes can be carried on (with the property \eqref{3-74}) as long as the reflection of $\overline{\Sigma^{\ast}_{\lambda}}$ is still contained in $\Omega^{\ast}$. We omit the details here.

Therefore, there exists a $\delta_{\ast}>0$ depending only on $x^{0}$ and $\Omega$ such that, $\overline{u}(x^{\ast})$ is monotone increasing along the internal normal direction in the region
\begin{equation}\label{3-75}
  \overline{\Sigma_{\delta_{\ast}}}:=\left\{x^{\ast}\in\overline{\Omega^{\ast}}\,|\,0\leq\left(x^{\ast}-(x^{0})^{\ast}\right)\cdot\nu^{0}\leq\delta_{\ast}\right\}.
\end{equation}
Since $\partial\Omega^{\ast}$ is $C^{n-2}$, there exists a small $0<\varepsilon_{1}<\min\{\frac{\delta_{\ast}}{8},\varepsilon_{0}\}$ depending on $x^{0}$ and $\Omega$ such that, for any $x^{\ast}\in B_{\varepsilon_{1}}\big((x^{0})^{\ast}\big)\cap\partial\Omega^{\ast}$, $\overline{u}(x^{\ast})$ is monotone increasing along the internal normal direction at $x^{\ast}$ in the region
\begin{equation}\label{3-76}
  \overline{\Sigma_{x^{\ast}}}:=\left\{z^{\ast}\in\overline{\Omega^{\ast}}\,\Big|\,0\leq(z^{\ast}-x^{\ast})\cdot\nu_{x^{\ast}}\leq\frac{3}{4}\delta_{\ast}\right\}.
\end{equation}
where $\nu_{x^{\ast}}$ denotes the unit internal normal vector at the point $x^{\ast}$ ($\nu_{(x^{0})^{\ast}}:=\nu^{0}$). Since $B_{\varepsilon_{1}}\big((x^{0})^{\ast}\big)\cap\partial\Omega^{\ast}$ is strictly convex, there exists a $\theta>0$ depending on $x^{0}$ and $\Omega$ such that
\begin{equation}\label{3-77}
  S:=\left\{\nu^{\ast}\in\mathbb{R}^n\,|\,|\nu^{\ast}|=1, \, \nu^{\ast}\cdot\nu^{0}\geq\cos\theta\right\}\subset\left\{\nu_{x^{\ast}}\,|\,x^{\ast}\in B_{\varepsilon_{1}}\big((x^{0})^{\ast}\big)\cap\partial\Omega^{\ast}\right\},
\end{equation}
and hence, it follows that, for any $x^{\ast}\in B_{\varepsilon_{1}}\big((x^{0})^{\ast}\big)\cap\partial\Omega^{\ast}$ and $\nu^{\ast}\in S$,
\begin{equation}\label{3-78}
  \overline{u}(x^{\ast}+s\nu^{\ast}) \quad \text{is monotone increasing with respect to} \,\,\, s\in\left[0,\frac{\delta_{\ast}}{2}\right].
\end{equation}
Now, let
\begin{equation}\label{3-79}
  D^{\ast}:=\left\{x^{\ast}+\varepsilon_{1}\nu^{0}\,|\,x^{\ast}\in B_{\varepsilon_{1}}\big((x^{0})^{\ast}\big)\cap\partial\Omega^{\ast}\right\},
\end{equation}
one immediately has
\begin{equation}\label{3-80}
  \max_{\overline{B_{\varepsilon_{1}}((x^{0})^{\ast})\cap\Omega^{\ast}}}\overline{u}(x^{\ast})\leq\max_{\overline{D^{\ast}}}\overline{u}(x^{\ast}).
\end{equation}
For any $x^{\ast}\in\overline{D^{\ast}}$, let
\begin{equation}\label{3-81}
  \overline{V_{x^{\ast}}}:=\left\{x^{\ast}+\nu^{\ast}\,\Big|\,\nu^{\ast}\cdot\nu^{0}\geq|\nu^{\ast}|\cos\theta, \, |\nu^{\ast}|\leq\frac{\delta_{\ast}}{4}\right\}
\end{equation}
be a piece of cone with vertex at $x^{\ast}$, then it is obvious that
\begin{equation}\label{3-82}
  \overline{u}(x^{\ast})=\min_{z^{\ast}\in\overline{V_{x^{\ast}}}}\overline{u}(z^{\ast}).
\end{equation}

Therefore, by \eqref{3-82} and Lemma \ref{lemma2}, we get, for any $x^{\ast}\in\overline{D^{\ast}}$,
\begin{eqnarray}\label{3-83}
 && C(\lambda_{1},p,\Omega)\geq\int_{\Omega}u^{p}(x)\phi(x)dx \\
 \nonumber &=&\int_{\Omega^{\ast}}\frac{\overline{u}^{p}(x^{\ast})}{|x^{\ast}-\overline{x^{0}}|^{2n-p(n-2)}}
  \phi\left(\frac{x^{\ast}-\overline{x^{0}}}{|x^{\ast}-\overline{x^{0}}|^2}+\overline{x^{0}}\right)dx^{\ast} \\
 \nonumber &\geq& \int_{\overline{V_{x^{\ast}}}}\frac{\overline{u}^{p}(z^{\ast})}{|z^{\ast}-\overline{x^{0}}|^{2n-p(n-2)}}
  \phi\left(\frac{z^{\ast}-\overline{x^{0}}}{|z^{\ast}-\overline{x^{0}}|^2}+\overline{x^{0}}\right)dz^{\ast}\\
 \nonumber &\geq& \overline{u}^p(x^{\ast})R^{2n-p(n-2)}_{0}|\overline{V_{x^{\ast}}}|\cdot\min_{\overline{\Omega^{r_{1}/2}}}\phi
 =:\overline{u}^{p}(x^{\ast})\cdot C(n,p,x^{0},\Omega),
\end{eqnarray}
where $\overline{\Omega^{r_{1}/2}}:=\{x\in\Omega\,|\,dist(x,\partial\Omega)\geq \frac{r_{1}}{2}\}$ with $r_{1}=\varepsilon_{1}R^{2}_{0}$, and hence
\begin{equation}\label{3-84}
  \overline{u}(x^{\ast})\leq C(n,p,x^{0},\lambda_{1},\Omega), \quad\quad \forall x\in\overline{D^{\ast}}.
\end{equation}
As a consequence, we derive that
\begin{equation}\label{3-85}
  \max_{\overline{B_{\varepsilon_{1}}((x^{0})^{\ast})\cap\Omega^{\ast}}}\overline{u}(x^{\ast})\leq\max_{\overline{D^{\ast}}}\overline{u}(x^{\ast})\leq C(n,p,x^{0},\lambda_{1},\Omega).
\end{equation}
There exists a small $r_{0}>0$ depending only on $x^{0}$ and $\Omega$ such that, for each $x\in\overline{B_{r_{0}}(x^{0})\cap\Omega}$, one has $x^{\ast}\in\overline{B_{\varepsilon_{1}}\big((x^{0})^{\ast}\big)\cap\Omega^{\ast}}$. Therefore, \eqref{3-85} yields
\begin{eqnarray}\label{3-86}
  \max_{\overline{B_{r_{0}}(x^{0})\cap\Omega}}u(x)&=&\max_{x\in\overline{B_{r_{0}}(x^{0})\cap\Omega}}|x^{\ast}-\overline{x^{0}}|^{n-2}\overline{u}(x^{\ast}) \\
  \nonumber &\leq& \frac{1}{R^{n-2}_{0}}\max_{\overline{B_{\varepsilon_{1}}((x^{0})^{\ast})\cap\Omega^{\ast}}}\overline{u}(x^{\ast})\leq C(n,p,x^{0},\lambda_{1},\Omega).
\end{eqnarray}

Since $x^{0}\in\partial\Omega$ is arbitrary and $\partial\Omega$ is compact, we can cover $\partial\Omega$ by finite balls $\{B_{r_{k}}(x^{k})\}_{k=0}^{K}$ with centers $\{x^{k}\}_{k=0}^{K}\subset\partial\Omega$ ($K$ depends only on $\Omega$). Therefore, there exists a $\bar{\delta}>0$ depending only on $\Omega$ such that
\begin{equation}\label{3-87}
  \|u\|_{L^{\infty}(\overline{\Omega}_{\bar{\delta}})}\leq\max_{0\leq k\leq K}\max_{\overline{B_{r_{k}}(x^{k})\cap\Omega}}u(x)\leq\max_{0\leq k\leq K}C(n,p,x^{k},\lambda_{1},\Omega)=:C(n,p,\lambda_{1},\Omega),
\end{equation}
where the boundary layer $\overline{\Omega}_{\bar{\delta}}:=\{x\in\overline{\Omega}\,|\,dist(x,\partial\Omega)\leq\bar{\delta}\}$. This completes the proof of boundary layer estimates under assumption ii).

This concludes our proof of Theorem \ref{Boundary}.
\end{proof}

\subsection{Blowing-up analysis and interior estimates}

In this subsection, we will obtain the interior estimates (and hence, global a priori estimates) via the blowing-up analysis arguments (for related literatures on blowing-up methods, please refer to \cite{BC,BM,CL3,CL4,CY0,Li,SZ1}).

Suppose on the contrary that Theorem \ref{Thm1} does not hold. By the boundary layer estimates (Theorem \ref{Boundary}), there exists a sequence of positive solutions $\{u_{k}\}\subset C^{n}(\Omega)\cap C^{n-2}(\overline{\Omega})$ to the critical order Navier problem \eqref{tNavier} and a sequence of interior points $\{x^{k}\}\subset\Omega\setminus\overline{\Omega}_{\bar{\delta}}$ such that
\begin{equation}\label{32-1}
  m_{k}:=u_{k}(x^{k})=\|u_{k}\|_{L^{\infty}(\overline{\Omega})}\rightarrow+\infty \quad \text{as} \,\, k\rightarrow\infty.
\end{equation}
For $x\in\Omega_{k}:=\{x\in\mathbb{R}^{n}\,|\,\lambda_{k}x+x^{k}\in\Omega\}$, we define
\begin{equation}\label{32-2}
  v_{k}(x):=\frac{1}{m_{k}}u_{k}(\lambda_{k}x+x^{k}) \quad \text{with} \,\, \lambda_{k}:=m_{k}^{\frac{1-p}{n}}\rightarrow0 \quad \text{as} \,\, k\rightarrow\infty.
\end{equation}
Then $v_{k}(x)$ satisfies $\|v_{k}\|_{L^{\infty}(\overline{\Omega_{k}})}=v_{k}(0)=1$ and
\begin{eqnarray}\label{32-3}
  (-\Delta)^{\frac{n}{2}}v_{k}(x)&=&\frac{1}{m_{k}}\lambda^{n}_{k}(-\Delta)^{\frac{n}{2}}u_{k}(\lambda_{k}x+x^{k}) \\
 \nonumber &=& \frac{1}{m_{k}}\lambda^{n}_{k}\left(u^{p}_{k}(\lambda_{k}x+x^{k})+t\right)=v^{p}_{k}(x)+\frac{t}{m^{p}_{k}}
\end{eqnarray}
for any $x\in\Omega_{k}$. Since $dist(x^{k},\partial\Omega)>\bar{\delta}$, one has
\begin{equation}\label{32-4}
  \Omega_{k}\supset\left\{x\in\mathbb{R}^{n}\,|\,|\lambda_{k}x|\leq\bar{\delta}\right\}=\overline{B_{\frac{\bar{\delta}}{\lambda_{k}}}(0)},
\end{equation}
and hence
\begin{equation}\label{32-5}
  \Omega_{k}\rightarrow\mathbb{R}^{n} \quad \text{as} \,\, k\rightarrow\infty.
\end{equation}

For arbitrary $x^{0}\in\mathbb{R}^{n}$, there exists a $N_{1}>0$, such that $\overline{B_{1}(x^{0})}\subset\Omega_{k}$ for any $k\geq N_{1}$. By \eqref{32-3} and $\|v_{k}\|_{L^{\infty}(\overline{\Omega_{k}})}\leq1$, we can infer from regularity theory and Sobolev embedding that
\begin{equation}\label{32-6}
  \|v_{k}\|_{C^{n-1,\gamma}(\overline{B_{1}(0)})}\leq C(1+t),
\end{equation}
and further that
\begin{equation}\label{32-7}
   \|v_{k}\|_{C^{2(n-1),\gamma}(\overline{B_{1}(0)})}\leq C(1+t)
\end{equation}
for $k\geq N_{1}$, where $0\leq\gamma<1$. As a consequence, by Arzel\`{a}-Ascoli Theorem, there exists a subsequence $\{v^{(1)}_{k}\}\subset\{v_{k}\}$ and a function $v\in C^{n}(\overline{B_{1}(x^{0})})$ such that
\begin{equation}\label{32-8}
  v^{(1)}_{k}\rightrightarrows v \quad \text{and} \quad (-\Delta)^{\frac{n}{2}}v^{(1)}_{k}\rightrightarrows(-\Delta)^{\frac{n}{2}}v \quad\quad \text{in} \,\, \overline{B_{1}(x^{0})}.
\end{equation}
There also exists a $N_{2}>0$ such that $\overline{B_{2}(x^{0})}\subset\Omega_{k}$ for any $k\geq N_{2}$. By \eqref{32-3} and $\|v_{k}\|_{L^{\infty}(\overline{\Omega_{k}})}\leq1$, we can deduce that
\begin{equation}\label{32-9}
  \|v^{(1)}_{k}\|_{C^{2(n-1),\gamma}(\overline{B_{2}(0)})}\leq C(1+t)
\end{equation}
for $k\geq N_{2}$, where $0\leq\gamma<1$. Therefore, by Arzel\`{a}-Ascoli Theorem again, there exists a subsequence $\{v^{(2)}_{k}\}\subset\{v^{(1)}_{k}\}$ and $v\in C^{n}(\overline{B_{2}(x^{0})})$ such that
\begin{equation}\label{32-10}
  v^{(2)}_{k}\rightrightarrows v \quad \text{and} \quad (-\Delta)^{\frac{n}{2}}v^{(2)}_{k}\rightrightarrows(-\Delta)^{\frac{n}{2}}v \quad\quad \text{in} \,\, \overline{B_{2}(x^{0})}.
\end{equation}
Continuing this way, for any $j\in\mathbb{N}^{+}$, we can extract a subsequence $\{v^{(j)}_{k}\}\subset\{v^{(j-1)}_{k}\}$ and find a function $v\in C^{n}(\overline{B_{j}(x^{0})})$ such that
\begin{equation}\label{32-11}
  v^{(j)}_{k}\rightrightarrows v \quad \text{and} \quad (-\Delta)^{\frac{n}{2}}v^{(j)}_{k}\rightrightarrows(-\Delta)^{\frac{n}{2}}v \quad\quad \text{in} \,\, \overline{B_{j}(x^{0})}.
\end{equation}
By extracting the diagonal sequence, we finally obtain that the subsequence $\{v^{(k)}_{k}\}$ satisfies
\begin{equation}\label{32-12}
  v^{(k)}_{k}\rightrightarrows v \quad \text{and} \quad (-\Delta)^{\frac{n}{2}}v^{(k)}_{k}\rightrightarrows(-\Delta)^{\frac{n}{2}}v \quad\quad \text{in} \,\, \overline{B_{j}(x^{0})}
\end{equation}
for any $j\geq1$. Therefore, we get from \eqref{32-3} that $0\leq v\in C^{n}(\mathbb{R}^{n})$ satisfies
\begin{equation}\label{32-13}
  (-\Delta)^{\frac{n}{2}}v(x)=v^{p}(x) \quad\quad \text{in} \,\, \mathbb{R}^{n}.
\end{equation}
By the Liouville theorem (Theorem \ref{Thm0}), we must have $v\equiv0$ in $\mathbb{R}^{n}$, which is a contradiction with
\begin{equation}\label{32-14}
  v(0)=\lim_{k\rightarrow\infty}v^{(k)}_{k}(0)=1.
\end{equation}

This concludes our proof of Theorem \ref{Thm1}.

\section{Proof of Theorem \ref{Thm2}}

In this section, by applying the a priori estimates (Theorem \ref{Thm1}) and the following Leray-Schauder fixed point theorem (see e.g. \cite{CLM}), we will prove the existence of positive solutions to the critical order Lane-Emden equations \eqref{Navier} with Navier boundary conditions.
\begin{thm}\label{L-S}
Suppose that $X$ is a real Banach space with a closed positive cone $P$, $U\subset P$ is bounded open and contains $0$. Assume that there exists $\rho>0$ such that $B_{\rho}(0)\cap P\subset U$ and that $K:\,\overline{U}\rightarrow P$ is compact and satisfies
\vskip 5pt
\noindent i) For any $x\in P$ with $|x|=\rho$ and any $\lambda\in[0,1)$, $x\neq\lambda Kx$;
\vskip 3pt
\noindent ii) There exists some $y\in P\setminus\{0\}$ such that $x-Kx\neq ty$ for any $t\geq0$ and $x\in\partial U$.
\vskip 5pt
\noindent Then, $K$ possesses a fixed point in $\overline{U_{\rho}}$, where $U_{\rho}:=U\setminus B_{\rho}(0)$.
\end{thm}

Now we let
\begin{equation}\label{4-1}
  X:=C^{0}(\overline{\Omega}) \quad\quad \text{and} \quad\quad P:=\{u\in X \,|\, u\geq0\}.
\end{equation}
Define
\begin{equation}\label{4-2}
  K(u)(x):=\int_{\Omega}G_{2}(x,y^{\frac{n}{2}})\int_{\Omega}G_{2}(y^{\frac{n}{2}},y^{\frac{n}{2}-1})\int_{\Omega}\cdots\int_{\Omega}G_{2}(y^{2},y^{1})u^{p}(y^{1})
  dy^{1}dy^{2}\cdots dy^{\frac{n}{2}},
\end{equation}
where $G_{2}(x,y)$ is the Green's function for $-\Delta$ with Dirichlet boundary condition in $\Omega$. Suppose $u\in C^{0}(\overline{\Omega})$ is a fixed point of $K$, i.e., $u=Ku$, then it is easy to see that $u\in C^{n}(\Omega)\cap C^{n-2}(\overline{\Omega})$ and satisfies the Navier problem
\begin{equation}\label{4-3}\\\begin{cases}
(-\Delta)^{\frac{n}{2}}u(x)=u^{p}(x) \,\,\,\,\,\,\,\,\,\, \text{in} \,\,\, \Omega, \\
u(x)=-\Delta u(x)=\cdots=(-\Delta)^{\frac{n}{2}-1}u(x)=0 \,\,\,\,\,\,\,\, \text{on} \,\,\, \partial\Omega.
\end{cases}\end{equation}

Our goal is to show the existence of a fixed point for $K$ in $P\setminus B_{\rho}(0)$ for some $\rho>0$ (to be determined later) by using Theorem \ref{L-S}. To this end, we need to verify the two conditions i) and ii) in Theorem \ref{L-S} separately.

\emph{i)} First, we show that there exists $\rho>0$ such that for any $u\in\partial B_{\rho}(0)\cap P$ and $0\leq\lambda<1$,
\begin{equation}\label{4-4}
  u-\lambda K(u)\neq0.
\end{equation}
For any $x\in\overline{\Omega}$, it holds that
\begin{eqnarray}\label{4-5}
 \nonumber |K(u)(x)|&=&\left|\int_{\Omega}G_{2}(x,y^{\frac{n}{2}})\int_{\Omega}G_{2}(y^{\frac{n}{2}},y^{\frac{n}{2}-1})\cdots\int_{\Omega}G_{2}(y^{2},y^{1})u^{p}(y^{1})
  dy^{1}\cdots dy^{\frac{n}{2}}\right| \\
  &\leq& \int_{\Omega}G_{2}(x,y^{\frac{n}{2}})\int_{\Omega}G_{2}(y^{\frac{n}{2}},y^{\frac{n}{2}-1})\cdots\int_{\Omega}G_{2}(y^{2},y^{1})dy^{1}\cdots dy^{\frac{n}{2}}\cdot\|u\|^{p}_{C^{0}(\overline{\Omega})} \\
 \nonumber &\leq& \rho^{p-1}\left\|\int_{\Omega}G_{2}(x,y)dy\right\|^{\frac{n}{2}}_{C^{0}(\overline{\Omega})}\cdot\|u\|_{C^{0}(\overline{\Omega})}.
\end{eqnarray}
Let $h(x):=\int_{\Omega}G_{2}(x,y)dy$, then it solves
\begin{equation}\label{4-6}\\\begin{cases}
-\Delta_{x}h(x)=1 \,\,\,\,\,\,\,\,\,\, \text{in} \,\,\, \Omega, \\
h(x)=0 \,\,\,\,\,\,\,\, \text{on} \,\,\, \partial\Omega.
\end{cases}\end{equation}
For a fixed point $x^{0}\in\Omega$, we define the function
\begin{equation}\label{4-7}
  \zeta(x):=\frac{(diam\,\Omega)^{2}}{2n}\left(1-\frac{|x-x^{0}|^{2}}{(diam\,\Omega)^{2}}\right)_{+},
\end{equation}
then it satisfies
\begin{equation}\label{4-8}\\\begin{cases}
-\Delta_{x}\zeta(x)=1 \,\,\,\,\,\,\,\,\,\, \text{in} \,\,\, \Omega, \\
\zeta(x)>0 \,\,\,\,\,\,\,\, \text{on} \,\,\, \partial\Omega.
\end{cases}\end{equation}
By maximum principle, we get
\begin{equation}\label{4-9}
  0\leq h(x)<\zeta(x)\leq\frac{(diam\,\Omega)^{2}}{2n}, \quad\quad \forall \,\, x\in\overline{\Omega}.
\end{equation}
Therefore, we infer from \eqref{4-5} and \eqref{4-9} that
\begin{equation}\label{4-10}
  \|K(u)\|_{C^{0}(\overline{\Omega})}<\rho^{p-1}\frac{(diam\,\Omega)^{n}}{(2n)^{\frac{n}{2}}}\|u\|_{C^{0}(\overline{\Omega})}=\|u\|_{C^{0}(\overline{\Omega})}
\end{equation}
if we take
\begin{equation}\label{4-11}
  \rho=\left(\frac{\sqrt{2n}}{diam\,\Omega}\right)^{\frac{n}{p-1}}>0.
\end{equation}
This implies that $u\neq\lambda K(u)$ for any $u\in\partial B_{\rho}(0)\cap P$ and $0\leq\lambda<1$.

\emph{ii)} Now, let $\eta\in C^{n}(\Omega)\cap C^{n-2}(\overline{\Omega})$ be the unique positive solution of
\begin{equation}\label{4-12}\\\begin{cases}
(-\Delta)^{\frac{n}{2}}\eta(x)=1, \,\,\,\,\,\,\,\,\,\, \,\,\, x\in\Omega, \\
\eta(x)=-\Delta\eta(x)=\cdots=(-\Delta)^{\frac{n}{2}-1}\eta(x)=0, \,\,\,\,\,\,\,\,\,\,\,\, x\in\partial\Omega.
\end{cases}\end{equation}
We will show that
\begin{equation}\label{4-13}
  u-K(u)\neq t\eta \quad\quad \forall \,\, t\geq0, \quad \forall u\in\partial U,
\end{equation}
where $U:=B_{R}(0)\cap P$ with sufficiently large $R>\rho$ (to be determined later). First, observe that for any $u\in\overline{U}$,
\begin{equation}\label{4-14}
  \left\|(-\Delta)^{\frac{n}{2}}K(u)\right\|_{C^{0}(\overline{\Omega})}=\|u\|^{p}_{C^{0}(\overline{\Omega})}\leq R^{p},
\end{equation}
and hence
\begin{equation}\label{4-15}
  \|K(u)\|_{C^{0,\gamma}(\Omega)}\leq CR^{p} \quad\quad \forall \,\, 0<\gamma<1,
\end{equation}
thus $K:\, \overline{U}\rightarrow P$ is compact.

We use contradiction arguments to prove \eqref{4-13}. Suppose on the contrary that, there exists some $u\in\partial U$ and $t\geq0$ such that
\begin{equation}\label{4-16}
  u-K(u)=t\eta,
\end{equation}
then one has $\|u\|_{C^{0}(\overline{\Omega})}=R>\rho>0$, $u\in C^{n}(\Omega)\cap C^{n-2}(\overline{\Omega})$ and satisfies the Navier problem
\begin{equation}\label{4-17}\\\begin{cases}
(-\Delta)^{\frac{n}{2}}u(x)=u^{p}(x)+t, \,\,\,\,\,\,\,\,\, u(x)>0, \,\,\,\,\,\,\,\,\,\, x\in\Omega, \\
u(x)=-\Delta u(x)=\cdots=(-\Delta)^{\frac{n}{2}-1}u(x)=0, \,\,\,\,\,\,\,\,\,\,\,\, x\in\partial\Omega.
\end{cases}\end{equation}
Choose a constant $C_{1}>\lambda_{1}$. Since $u(x)>0$ in $\Omega$ and $p>1$, it is easy to see that, there exists another constant $C_{2}>0$ (e.g., take $C_{2}=C_{1}^{\frac{p}{p-1}}$), such that
\begin{equation}\label{4-18}
  u^{p}(x)\geq C_{1}u(x)-C_{2}.
\end{equation}
If $t\geq C_{2}$, then we have
\begin{equation}\label{4-19}
  (-\Delta)^{\frac{n}{2}}u(x)=u^{p}(x)+t\geq C_{1}u(x)-C_{2}+t\geq C_{1}u(x) \quad\quad \text{in} \,\, \Omega.
\end{equation}
Multiplying both side of \eqref{4-19} by the eigenfunction $\phi(x)$, and integrating by parts yield
\begin{eqnarray}\label{4-20}
  C_{1}\int_{\Omega}u(x)\phi(x)dx&\leq&\int_{\Omega}(-\Delta)^{\frac{n}{2}}u(x)\cdot\phi(x)dx=\int_{\Omega}u(x)\cdot(-\Delta)^{\frac{n}{2}}\phi(x)dx \\
 \nonumber &=&\lambda_{1}\int_{\Omega}u(x)\phi(x)dx,
\end{eqnarray}
and hence
\begin{equation}\label{4-21}
  0<(C_{1}-\lambda_{1})\int_{\Omega}u(x)\phi(x)dx\leq0,
\end{equation}
which is absurd. Thus, we must have $0\leq t<C_{2}$. By the a priori estimates (Theorem \ref{Thm1}), we know that
\begin{equation}\label{4-22}
  \|u\|_{L^{\infty}(\overline{\Omega})}\leq C(n,p,t,\lambda_{1},\Omega).
\end{equation}
We will show that the above a priori estimates are uniform with respect to $0\leq t<C_{2}$, i.e., for $0\leq t<C_{2}$,
\begin{equation}\label{4-23}
  \|u\|_{L^{\infty}(\overline{\Omega})}\leq C(n,p,C_{2},\lambda_{1},\Omega)=:C_{0}.
\end{equation}
Indeed, it is clear from Theorem \ref{Boundary} that, the thickness $\bar{\delta}$ of the boundary layer and the boundary layer estimates are uniform with respect to $t$. Therefore, if \eqref{4-23} does not hold, there exist sequences $\{t_{k}\}\subset[0,C_{2})$, $\{x^{k}\}\subset\Omega\setminus\overline{\Omega}_{\bar{\delta}}$ and $\{u_{k}\}$ satisfying
\begin{equation}\label{4-24}\\\begin{cases}
(-\Delta)^{\frac{n}{2}}u_{k}(x)=u_{k}^{p}(x)+t_{k}, \,\,\,\,\,\,\,\,\,\, \,\,\, x\in\Omega, \\
u_{k}(x)=-\Delta u_{k}(x)=\cdots=(-\Delta)^{\frac{n}{2}-1}u_{k}(x)=0, \,\,\,\,\,\,\,\,\,\,\,\, x\in\partial\Omega,
\end{cases}\end{equation}
but $m_{k}:=u_{k}(x^{k})=\|u_{k}\|_{L^{\infty}(\overline{\Omega})}\rightarrow+\infty$ as $k\rightarrow\infty$. For $x\in\Omega_{k}:=\{x\in\mathbb{R}^{n}\,|\,\lambda_{k}x+x^{k}\in\Omega\}$, we define $v_{k}(x):=\frac{1}{m_{k}}u_{k}(\lambda_{k}x+x^{k})$ with $\lambda_{k}:=m^{\frac{1-p}{n}}_{k}\rightarrow0$ as $k\rightarrow\infty$. Then $v_{k}(x)$ satisfies $\|v_{k}\|_{L^{\infty}(\overline{\Omega_{k}})}=v_{k}(0)=1$ and
\begin{equation}\label{4-25}
(-\Delta)^{\frac{n}{2}}v_{k}(x)=v^{p}_{k}(x)+\frac{t_{k}}{m^{p}_{k}}
\end{equation}
for any $x\in\Omega_{k}$. Since $0\leq t<C_{2}$ and $m_{k}\rightarrow+\infty$, by completely similar blowing-up methods as in the proof of Theorem \ref{Thm1} in subsection 3.2, we can also derive a subsequence $\{v^{(k)}_{k}\}\subset\{v_{k}\}$ and a function $v\in C^{n}(\mathbb{R}^{n})$ such that
\begin{equation}\label{4-26}
  v^{(k)}_{k}\rightrightarrows v \quad\quad \text{and} \quad\quad (-\Delta)^{\frac{n}{2}}v^{(k)}_{k}\rightrightarrows(-\Delta)^{\frac{n}{2}}v \quad\quad \text{in} \,\, \overline{B_{j}(x^{0})}
\end{equation}
for arbitrary $j\geq1$, and hence $0\leq v\in C^{n}(\mathbb{R}^{n})$ solves
\begin{equation}\label{4-27}
  (-\Delta)^{\frac{n}{2}}v(x)=v^{p}(x) \quad\quad \text{in} \,\, \mathbb{R}^{n}.
\end{equation}
By Theorem \ref{Thm0}, one immediately has $v\equiv0$, which contradicts with $v(0)=1$. Therefore, the uniform estimates \eqref{4-23} must hold.

Now we let $R:=C_{0}+\rho$ and $U:=B_{C_{0}+\rho}(0)\cap P$, then \eqref{4-23} implies
\begin{equation}\label{4-28}
  \|u\|_{L^{\infty}(\overline{\Omega})}\leq C_{0}<C_{0}+\rho,
\end{equation}
which contradicts with $u\in\partial U$. This implies that
\begin{equation}\label{4-29}
  u-K(u)\neq t\eta
\end{equation}
for any $t\geq0$ and $u\in\partial U$ with $U=B_{C_{0}+\rho}(0)\cap P$.

From Theorem \ref{L-S}, we deduce that there exists a $u\in\overline{\big(B_{C_{0}+\rho}(0)\cap P\big)\setminus B_{\rho}(0)}$ satisfies
\begin{equation}\label{4-30}
  u=K(u),
\end{equation}
and hence $\rho\leq\|u\|_{L^{\infty}(\overline{\Omega})}\leq C_{0}+\rho$ solves the critical order Navier problem
\begin{equation}\label{4-31}\\\begin{cases}
(-\Delta)^{\frac{n}{2}}u(x)=u^{p}(x), \,\,\,\,\,\,\, u(x)>0, \,\,\,\,\,\,\,\,\, x\in\Omega, \\
u(x)=-\Delta u(x)=\cdots=(-\Delta)^{\frac{n}{2}-1}u(x)=0, \,\,\,\,\,\,\,\,\,\,\,\, x\in\partial\Omega.
\end{cases}\end{equation}
By regularity theory, we can see that $u\in C^{n}(\Omega)\cap C^{n-2}(\overline{\Omega})$.

This concludes our proof of Theorem \ref{Thm2}.

\end{document}